\documentclass[10pt]{article}
\usepackage{amsfonts}
\usepackage{amsfonts}  
\usepackage{pdfsync}
\title{Principal eigenvalues for Fully non linear singular or degenerate operators in  punctured balls }

\author{  Fran\c{c}oise Demengel\\
  D\'epartement de Math\'ematiques,
CY Paris University
 }

\date{}

\usepackage{amsthm}
\usepackage{amsmath}
\usepackage{color}

\usepackage[colorlinks=true]{hyperref}

\catcode`@=11 \@addtoreset{equation}{section} \catcode`@=12

\newtheorem{theo}{Theorem}[section]
\newtheorem{prop}[theo]{Proposition}
\newtheorem{rema}[theo]{Remark}

\newtheorem{cor}[theo]{Corollary}
\newtheorem{lemme}[theo]{Lemma}

\def\R{\mathbb  R}

\newcommand{\N}{{\mathbb N}}

\begin{document}
\maketitle

\begin{abstract} This paper is devoted to the proof of the existence of  the principal eigenvalue and related eigenfunctions
for fully nonlinear degenerate or singular uniformly elliptic equations posed in a punctured ball, in presence of  a singular potential. More precisely,   we   analyze  existence, uniqueness  and regularity    of solutions $( \bar\lambda_\gamma, u_\gamma)$ of the equation 
$$| \nabla u |^\alpha F( D^2 u_\gamma)+ \bar \lambda_\gamma {u_\gamma^{1+\alpha}  \over r^\gamma} = 0\ {\rm in} \ B(0,1)\setminus \{0\}, \ u_\gamma = 0 \ {\rm on} \ \partial B(0,1)$$
 where $u_\gamma>0$ in $B(0,1)$, $\alpha >-1$ and $\gamma >0$. We  prove existence of radial solutions which are continuous on $\overline{ B(0,1)}$ in the case $\gamma <2+\alpha$,  existence of unbounded solutions which do ot satisfy the boundary condition in the case $\gamma = 2+\alpha $ and a non existence result for $\gamma >2+\alpha$. 
 We also give the explicit value of $\bar \lambda_{2+\alpha} $ in the case of  Pucci's operators,  which generalizes the  Hardy--Sobolev constant for the Laplacian, and the previous results of \cite{BDL}. 
\end{abstract}

\section{Introduction}

  We look for   \lq\lq radial eigenvalues"  and related positive eigenfunctions for the problem 
  $$| \nabla u |^\alpha F( D^2 u) +\mu  r^{-\gamma} u^{1+\alpha}  = 0 \,  \mbox{ in }\ \overline{B(0,1)} \setminus \{0\},$$
 where  $\gamma>0$, and $F$ is a second order, fully nonlinear, uniformly elliptic operator. $\alpha$ is supposed to be $> -1$ and the solutions are intented in the viscosity sense given by \cite{BD1}.  In the model cases, $F$ will be one of  Pucci's extremal operators. It is clear that, if $\alpha = 0$,  as in the case of   Laplace operator,  $\gamma = 2$  is a critical case, being $F$ a second  order differential operator. This case $\alpha = 0$ is studied in \cite{BDL} and the present paper extends most of the results in \cite{BDL} to the case $\alpha \neq 0$. As in the case of the Laplacian, for $\gamma <2+\alpha$  not  "too bad" solutions are expected to exist, whereas   there are no solutions if  $\gamma >2+\alpha $. In the borderline case $\gamma=2+\alpha $, we will prove existence of unbounded "eigenfunctions". 
 \smallskip

 The case of the Laplacian and $\gamma = 2$ is naturally linked  to Hardy's inequality: 
 let us suppose that $N>2$ and let $u\in H_0^1 ( B(0,1))$ (respectively, $u\in H^1( \R^N)$), then ${u(x)\over |x|}$ belongs to $L^2(B(0,1))$ (respectively ${u(x)\over |x|} \in L^2( \R^N)$), and there exists a positive constant $c$ such that 
  $$ \int\left(  { |u(x)| \over |x|}\right)^2 \leq c \int | \nabla u |^2.$$
 Furthermore, the best constant is  $c= {4\over ( N-2)^2}$  and it is not achieved, in the sense that 
  \begin{equation}\label{varlap} \inf _{ u\in H_0^1( B(0,1)), \int_{B(0,1)} \left({ |u(x) |\over |x|}|\right)^2 =1}  \int_{B(0,1)}  | \nabla u |^2= { (N-2)^2 \over 4}.\end{equation}
   and there is no $u\in H_0^1$ which realizes the infimum.   By obvious arguments, $B(0,1)$ can be replaced by any bounded regular open set of $\R^N$ containing $0$,  and the optimal constant does not depend on the size of the open set. Note that the right hand side of (\ref{varlap}) coincides with the variational characterization of  the "first" or "principal " eigenvalue for the equation 
   $$-\Delta u = \lambda { u \over |x|^2}.$$
   For further knowledge on the Hardy--Sobolev inequality and for the case of the $p$-Laplacian, we refer to  \cite{GAPA, St}.

   On the other hand, if we replace in the constraint  the power $2$ by an exponent $\gamma <2$, compactness from $H_0^1(B(0,1))$ into the weighted space $L^2 ( B(0,1), {1\over r^\gamma})$ holds, and this allows to obtain  existence of minima in  $H_0^1 ( B(0,1))$ by  standard arguments in calculus of variations. In that case, denoting 
 \begin{equation}\label{variational}
 \bar \lambda_\gamma =  \inf _{ u\in H_0^1( B(0,1)), \int_{B(0,1)}{|u(x)|^2\over |x|^\gamma} =1}  \int _{B(0,1)}| \nabla u |^2
 \end{equation}
    one sees that $\bar \lambda_\gamma$ is also the first eigenvalue for the equation
    $$\Delta u + \bar \lambda _\gamma {u\over r^\gamma}  = 0\, ,$$ meaning  that $\bar \lambda_\gamma$ is such that there exists $u>0$ in $H_0^1( B(0,1))$ satisfying the equation. 
    
     Note that, by its definition, $\bar \lambda_\gamma$ depends on the domain, since
     $$\bar \lambda_\gamma ( B(0, t) ) = {1 \over t^{ 2-\gamma}}\bar \lambda _\gamma ( B(0,1)).$$
     
      If $\gamma >2$ there is no embedding from  $H_0^1(B(0,1))$ into $L^2 ( B(0,1), {1\over r^\gamma})$. Indeed, as an example, the function $$u(r) = r^{-{N-2\over 2}+\epsilon}(-\log r)$$ with $0<\epsilon < { \gamma-2 \over 2}$, belongs to $H_0^1 ( B(0,1))$ and satisfies 
      $$ \int _{ B(0,1) }{ u(|x|)^2 \over |x|^\gamma} = +\infty.$$

Let us  now focus on  non variational cases. More precisely, let us suppose that $F$ is fully nonlinear uniformly elliptic operator, that is $F$ is a continuous function defined on the set $\mathcal{S}_N$ of symmetric $N\times N$ matrices, and it satisfies, for positive constants $\Lambda\geq \lambda>0$,
       \begin{equation}\label{FNL} \lambda\,  \hbox{tr}( M^\prime) \leq F( M+ M^\prime) -F( M) \leq \Lambda\,  \hbox{tr}(M^\prime)\, ,
       \end{equation}
for all $M, M'\in \mathcal{S}_N$, with $M'$ positive semidefinite.
       
      Suppose also that $F$ is rotationally invariant and positively homogeneous of degree $1$, i.e. 
       \begin{equation}\label{posh} {\rm for \ any \ M} \in { \cal S}\  {\rm  and } \ t>0,
       F( tM) = tF(M).\end{equation}
  In this case the first eigenvalue for the equation 
      $$F( D^2 u) + \lambda { u \over |x|^\gamma} =0\qquad \hbox{ in } B(0,1)\setminus \{0\}$$
       can be defined  on the model  of \cite{BNV}, i.e. by the optimization formula 
        \begin{equation}\label{deflambdagamma1} \bar \lambda_\gamma = \bar \lambda_\gamma (F, r^{-\gamma})= \sup \{\mu\, : \  \exists\,   
        u \in C(B(0,1)\setminus \{0\})\, ,\ u>0 \hbox{ in } B(0,1),  \ \ F (D^2 u) + \mu { u \over r^\gamma} \leq 0\}\, ,
        \end{equation} 
where the differential inequality appearing above is understood in the viscosity sense.   
In the  model cases, the operator $F$ will be  one of Pucci's extremal operators. Let us recall their definition: by decomposing each matrix $M\in \mathcal{S}_N$ as $M = M^+-M^-$, where $M^+$ and $M^-$ are positive semidefinite matrices satisfying  $M^+M^-=O$, then Pucci's sup operator can be defined as
$$\mathcal{M}^+ _{\lambda, \Lambda} (M) = \Lambda \hbox {tr}( M^+)- \lambda \hbox {tr}( M^-)\, ,$$
 as well as Pucci's inf operator is given by
           $$\mathcal{M}^-_{\lambda, \Lambda} (M)= \lambda \hbox {tr}( M^+)- \Lambda \hbox {tr}( M^-)= -\mathcal{M}^+_{\lambda, \Lambda}  (-M).$$

 In the sequel, we will omit in the notation the dependence on  the ellipticity constants, which are fixed once for all.    
 
 We further recall that for a   $C^2$  radial function $u(x) = u(|x|)$, one has
                      $$ D^2 u(x)  = u^{\prime \prime} (r){ x\otimes x \over r^2} + {u^\prime (r) \over r} \left( I-{x\otimes x \over r^2}\right)\, ,$$
and, as a consequence,  
            $${\cal M}^+ ( D^2 u) =\Lambda (N-1)\left({u^\prime (r) \over r}\right)^+- \lambda (N-1)\left({u^\prime (r) \over r}\right)^-+ \Lambda (u^{\prime \prime}(r))^+-\lambda  ( u^{\prime \prime}(r) )^-\, ,$$
 $${\cal M}^- (D^2 u) =\lambda (N-1)\left({u^\prime (r) \over r}\right)^+- \Lambda (N-1)\left({u^\prime (r) \over r}\right)^-+ \lambda (u^{\prime \prime}(r))^+-\Lambda  ( u^{\prime \prime}(r) )^-\, .$$
 Thus,  the ODEs satisfied by radial solutions  of Pucci's extremal equations have coefficients depending on the  dimension like parameters, associated with $\mathcal{M}^+$ and $\mathcal{M}^-$ respectively, defined as     
$$\tilde N_+ = { \lambda \over \Lambda} ( N-1) +1\, ,\quad \tilde N_- = { \Lambda \over \lambda}( N-1)+1\, .$$
 Note that one has always 
$
\tilde{N}_-\geq N \geq \tilde{N}_+\, ,
$
with equalities holding true if and only if $\Lambda=\lambda$. We will  assume always that $\tilde{N}_+>2$. This will imply in particular that since $\gamma < \alpha+2$, $(\tilde N_+-1)(1+\alpha) + 1-\gamma >0$. 

       We recall the results obtained in \cite{BDL}. 
       
       \begin{theo}Suppose that $F$ satisfies \eqref{FNL}, \eqref{posh} and is rotationaly invariant. Suppose that 
 $\gamma <2$.    Then:
 \begin{itemize}   
\item[(i)] $\bar \lambda_\gamma $ defined in (\ref{deflambdagamma1} ) is positive and there exists a  function $u$, continuous in  $\overline{B(0,1)}$, radial, strictly positive in $B(0,1)$, such that 
     $$\left\{ \begin{array}{cl}
     F ( D^2 u) + \bar \lambda_\gamma { u \over r^\gamma} = 0 & \hbox{ in } \ B(0,1)\setminus \{0\}\\[1ex]
       u = 0 & \hbox{ on } \ \partial B(0,1)
       \end{array}\right..$$
       Furthermore $u$ is $C^2(B(0,1)\setminus\{0\})$ and it can be extended on  $B(0,1)$ as a Lipschitz continuous function if $\gamma \leq 1$, as a function of class $C^1(B(0,1))$ when $\gamma <1$,  and as an H\"older continuous function with exponent $2-\gamma$ if $\gamma >1$. 

\item[(ii)] Assume that  $\gamma = 2$. Then: 
   For the operator $\mathcal{M}^+$ one has
$$ \bar \lambda_2(\mathcal{M}^+) =\Lambda { (\tilde N_+-2)^2\over 4}$$
and the function
$
u(x)= r^{-\frac{\tilde{N}_+-2}{2}} (-\ln r)
$
is an explicit solution of
$$\left\{ \begin{array}{cl}
 \mathcal{M}^+( D^2 u) + \bar \lambda_\gamma { u \over r^\gamma} = 0 & \hbox{ in } \ B(0,1)\setminus \{0\}\\
       u = 0 & \hbox{ on } \ \partial B(0,1)
       \end{array}\right.$$
  
 \item[(iii)]   If $\gamma >2$, then  the eigenvalue $\bar \lambda^\prime_\gamma$ defined by 
$$ \bar \lambda_\gamma^{\prime} := \sup \{\mu\, : \  \exists\,   
        u \in C^2(B(0,1)\setminus \{0\})\, ,\ u>0 \hbox{ in } B(0,1)\setminus \{0\}, \ u \hbox{ radial},  \ \ F (D^2 u) + \mu { u \over r^\gamma} \leq 0\}
$$
  satisfies    $\bar \lambda_\gamma^\prime = 0$.  
 \end{itemize}     \end{theo}
    
\bigskip

    We now describe the results here enclosed. As we said above,  we look for  positive "radial eigenvalues" for the equation 
  $$| \nabla u |^\alpha F( D^2 u) +\mu  r^{-\gamma} u ^{1+\alpha} = 0 \,  \mbox{ in }\ \overline{B(0,1)} \setminus \{0\},$$
  
   where $\alpha >-1$, $0< \gamma < \alpha+2$.  $F$ is always a fully non linear positively homogeneous operator,  ie, it satisfies \eqref{FNL}, \eqref{posh} , and is   rotationally invariant.     For $\gamma >0$ we define 
      \begin{equation}\label{deflambdagamma} \bar \lambda_\gamma = \sup \{\mu\, : \  \exists\,   
        u \in C(B(0,1)\setminus \{0\})\, ,\ u>0 \hbox{ in } B(0,1), \ {\rm radial}, \ | \nabla u |^\alpha F (D^2 u) + \mu { u^{1+\alpha}  \over r^\gamma} \leq 0\}.
        \end{equation}

      \begin{theo} \label{exigamma}
      Suppose that $F$ satisfies \eqref{FNL}, \eqref{posh} and is rotationaly invariant,  and that $\alpha >-1$. 
Suppose that $0< \gamma <2+\alpha $.    Then:
 \begin{itemize}   
\item[(i)] $\bar \lambda_\gamma $ defined in (\ref{deflambdagamma} ) is positive and there exists a  function $u$, continuous in  $\overline{B(0,1)}$, radial, strictly positive in $B(0,1)$, such that 
     $$\left\{ \begin{array}{cl}
     | \nabla u |^\alpha F ( D^2 u) + \bar \lambda_\gamma { u^{1+\alpha}  \over r^\gamma} = 0 & \hbox{ in } \ B(0,1)\setminus \{0\}\\[1ex]
       u = 0 & \hbox{ on } \ \partial B(0,1)
       \end{array}\right..$$
       Furthermore $u$ is ${ \cal C}^1(B(0,1)\setminus\{0\})$,  $|u^\prime |^\alpha u^\prime \in { \cal C}^1(B(0,1)\setminus\{0\})$, $u$ is Lipschitz continuous on $\overline{B(0,1)}$ in the case $\gamma \leq 1$, ${\cal C}^1(B(0,1))$ when $\gamma <1$,  and H\"older continuous of exponent ${2+\alpha -\gamma\over 1+\alpha} $ if $\gamma >1$. 
       \end{itemize}
       \end{theo}
       \begin{rema} In fact we will prove that $u\in { \cal C}^2( B(0,1)\setminus\{0\}$. 
       \end{rema} 
             
       Statement (i) of the  above theorem shows in particular that $\bar \lambda_\gamma$ is actually achieved on smooth radial eigenfunctions. Thus,  
      if we define 
\begin{equation}\label{lambdaprime}\begin{array}{lc}
 \bar \lambda_\gamma^\prime &= \sup \{\mu\, : \  \exists\,   
        u \in C(B(0,1)\setminus \{0\})\, ,\ u>0 \hbox{ in } B(0,1), \ {\rm radial}, \ u \in {\cal C}^1(B(0,1)\setminus\{0\}), \\
        & |u^\prime |^\alpha u^\prime \in { \cal C}^1 (B(0,1)\setminus\{0\}),         \  \ | \nabla u |^\alpha F (D^2 u) + \mu { u^{1+\alpha}  \over r^\gamma} \leq 0\},
      \end{array} 
      \end{equation}  it then follows that 
  $$\bar \lambda_\gamma = \bar \lambda_\gamma^\prime.$$
 Actually, we will work initially with the smooth eigenvalue $\bar \lambda_\gamma'$, and we will finally prove that it coincides with $\bar \lambda_\gamma$.

       \begin{theo}\label{gamma=2+alpha}
 Suppose that $F = {\cal M}^+$ and that  $\gamma = 2+\alpha $. Then: 
\begin{itemize}
\item[(i)]    $$ \bar \lambda_{2+\alpha}  ={\Lambda\over \alpha+1} \left( { (\tilde N_+-2)(1+\alpha)\over 2+\alpha}\right)^{2+\alpha} $$
   Some  "eigenfunction" is $r^{-\tau}$ with 
     $\tau = {( \tilde N_+-2)(\alpha+1)\over 2+\alpha}$.

\item[(ii)]  $\bar \lambda_{2+\alpha} $ is stable under various  regularization 
     $$\bar \lambda_{2+\alpha}  = \lim_{ \gamma \rightarrow 2+\alpha} \bar \lambda_\gamma$$ 
      $$\bar \lambda_{2+\alpha}  = \lim_{ \delta \rightarrow 0 } \bar \lambda\left({1\over r^{2+\alpha} },  B(0,1) \setminus \overline{ B(0, \delta)}\right)$$
      $$\bar \lambda_{2+\alpha}  = \lim_{ \epsilon \rightarrow 0} \bar \lambda \left( {1 \over ( r^2+\epsilon^2)^{\alpha+2\over 2}}, B(0,1)\right).$$

\end{itemize}
    \end{theo}

      We will finally  prove

        \begin{theo}\label{gamma>}
        Suppose that $F$ satisfies \eqref{FNL}, \eqref{posh} and is rotationaly invariant.  Suppose that $\alpha >-1$. 
         If $\gamma > \alpha+2$, $\bar \lambda_\gamma^\prime = 0$.
         \end{theo}

\section{The case $\gamma < \alpha+2$}
\subsection{Maximum principles, existence and regularity results}

For simplicity in the reading let us introduce   the space invoked above 
\begin{equation}\label{C1alpha}
 {\cal C}_{1, \alpha} = \{ u \in { \cal C}^1( B(0,1)\setminus\{0\},\ {\rm radial}\  \ |u^\prime|^\alpha u^\prime \in { \cal C}^1( B(0,1)\setminus\{0\}\}.
 \end{equation} 
\begin{rema}\label{remC1alpha}
Let us observe that the radial   solutions of $| \nabla u |^\alpha F(D^2 u) = {f\over r^\gamma}$ satisfy  $u\in { \cal C}_{1, \alpha}$. Indeed,  let us  use the notations in \cite{GPL}, \cite{EFQ}  : 
$${\cal F} (r, l, m) = F(  lId+ (m-l) {x\otimes x \over r^2}))$$ Then the equation above can be written $${\cal F} ( r, |u^\prime |^\alpha u^\prime,   ( |u^\prime |^\alpha u^\prime)^\prime ) = {f \over r^\gamma}$$
 and then using the Ellipticity of $F$ , adapting Lemma 2.1 in \cite{EFQ}, (see also \cite{GPL} ),   there exists ${\cal G}$ locally Lipschitz in $ \R^2$ so that 
 $$( |u^\prime |^\alpha u^\prime )^\prime   = { \cal G} ( |u^\prime|^\alpha u^\prime r^{-1}, {f(r) \over r^\gamma}).$$
  As a consequence the solutions of such equations satisfy 
  $ u\in { \cal C}_{1, \alpha}$.  Furthermore if $u^\prime \neq 0$, we get that $u^{\prime \prime} \in { \cal C}( B(0,1)\setminus\{0\})$, dividing  the equation defining ${ \cal G}$ by $|u^\prime |^\alpha$. This  will  be the case for  an eigenfunction, since we will  observe that it satisfies $u^\prime <0$ in $r>0$. 
     \end{rema} 

 The first results of the present section  are  two  crucial technical lemmata .
    
   \begin{lemme}\label{fabiana1}
Let $ f\in C\left( B(0,1)\setminus \{0\}\right)$ be a radial, bounded and  positive function and assume that  $u\in C_{1, \alpha} $ is  a radial,  bounded function   which satisfies  
\begin{equation}\label{subsol}| \nabla u |^\alpha {\cal M}^+ ( D^2 u) \geq  f   r^{-\gamma}\qquad \hbox{ in } B(0,1)\setminus \{0\}\, .\end{equation}
 Then 
\begin{itemize}
 \item[(i)]  $u^\prime \geq 0$ in a right neighborhood of $0$;
  \item[(ii)] $\displaystyle \lim_{r\rightarrow 0} u^\prime(r) r^{ {\tilde N}_--1} = 0$ and in a right neighborhood of $0$ one has
  $$ u'(r)\geq\left(   (1+\alpha)   \frac{\inf f}{\Lambda ((\tilde N_--1)(1+\alpha) + 1-\gamma)}\right)^{1\over 1+\alpha} r^{1-\gamma\over 1+\alpha}\, ;$$
 \item[(iii)] if,  furthermore, $| \nabla u |^\alpha { \cal M}^+ ( D^2 u) = f r^{-\gamma}$, then, in a   right neighborhood of $0$, one  also has 
                      $$u'(r) \leq \left( (1+\alpha) \frac{\sup f}{\lambda ((N-1)(1+\alpha) +1-\gamma)}\right)^{1\over 1+\alpha}  r^{1-\gamma\over 1+\alpha}\, .$$
In particular, there exists a constant $c>0$ such that, for $r$ sufficiently small,
                       $$|u^\prime (r) | \leq c r^{1 -\gamma\over 1+\alpha}$$
                        and then $u$ can be extended as a  locally Lipschitz continuous function  in $B(0,1)$ if $\gamma \leq 1$,  it belongs to  $C^1(B(0,1))$ if $\gamma <1$, and it is locally H\"older continuous in $B(0,1)$ with exponent ${2+\alpha -\gamma\over 1+\alpha}$ if $\gamma >1$. 
 \end{itemize}
                        \end{lemme} 
                        
                      For the convenience of the reading    we define $G_{\Lambda, \lambda}  ( p, q) =  \Lambda (p^++ q^+)-\lambda (p^-+ q^-)$. Note that 
                         $G_{\lambda,\Lambda}$ is subadditive .

                        \begin{proof} 
                        Note that  when $u$ is radial, 
                      $$ | \nabla u |^\alpha {\cal M}^+ ( D^2 u)  =  G_{\lambda, \Lambda} ( (1+\alpha)^{-1} (|u^\prime |^\alpha u^\prime )^\prime, ({ N-1\over r} ) |u^\prime |^\alpha u^\prime )$$  Since $u$ is ${ \cal C}^1$ and $|u^\prime |^\alpha u^\prime$ is ${ \cal C}^1$, suppose that $u^\prime$ changes sign around  $0$, then there exists a numerable set  $r_n$ so that  on $[r_{2n+1}, r_{2n}]$ $u^\prime \leq 0$, and on $[r_{2n+2}, r_{2n+1}[$ $u^\prime \geq 0$. Then there exists $s_n$ in $]r_{2n+1}, r_{2n}[$ so that $(|u^\prime |^\alpha u^\prime )^\prime (s_n) = 0$ a contradiction with the inequation satisfied by $u^\prime$. Then $u^\prime$ does not change sign and if we had $u^\prime \leq 0$ then 
                      $$ \Lambda (|u^\prime |^\alpha u^\prime )^\prime (1+\alpha)^{-1} + \lambda |u^\prime |^\alpha u^\prime ( N-1) r^{-1} \geq  f r^{-\gamma}$$
                       which implies that 
                       $$( |u^\prime |^\alpha u^\prime r^{ (N_+-1)(1+\alpha) } )^\prime \geq (1+\alpha)\Lambda^{-1} f r^{ (N_+-1)(1+\alpha) -\gamma}$$ In particular $ |u^\prime |^\alpha u^\prime r^{ (N_+-1)(1+\alpha) }$ has a limit when $r$ goes to zero. Suppose that this limit is $<0$, then this would  implies  for some positive constant $c$
                       $$u^\prime \leq -c r^{1-N_+}$$ near zero, which contradicts $u$ bounded. Then integrating the equation above one obtains 
                       $$  |u^\prime |^\alpha u^\prime r^{ (N_+-1)(1+\alpha)}\geq 0$$ which yields a contradiction. We have obtained that $u^\prime \geq 0$  and then whatever is the sign of $u^{\prime \prime}$ one has                         

$$ (1+\alpha)^{-1} ( |u^\prime |^\alpha u^\prime )^\prime + { \Lambda(N-1) \over r  \lambda}  |u^\prime |^\alpha u^\prime  \geq f r^{-\gamma} \Lambda^{-1}$$
 which implies that 
 $$ (  |u^\prime |^\alpha u^\prime r^{ (N_--1)(1+\alpha)})^\prime \geq (1+\alpha) \Lambda^{-1} f r^{ -\gamma + (N_--1)(1+\alpha)} $$
Even if it is not necessary for the sequel, let us remark   in particular that  $ |u^\prime |^\alpha u^\prime r^{ (N_--1)(1+\alpha)}$ has a limit when $r$ goes to zero,  and  if the limit was $c>0$ this would  contradict once more that $u$ is bounded. Then $\lim_{r\to 0} |u^\prime |^\alpha u^\prime r^{ (N_--1)(1+\alpha)}=0$.    We obtain by integrating 
   $$  |u^\prime |^\alpha u^\prime r^{ (N_--1)(1+\alpha)}\geq (1+\alpha )\Lambda^{-1} \inf f  r^{ (N_--1)(1+\alpha)+1-\gamma}$$
      which yields the result.

     item (iii)    can be proved in the same fashion.
    
    \end{proof} 
    \begin{rema}\label{strict}
     From the arguments used in the proof of item 1, $f>0$ is sufficient to get $u^\prime >0$ near zero. 
     This will be used in the proof of the uniqueness  in Theorem \ref{exilambda}. . 
     \end{rema}
     \begin{rema}\label{lemmaformm}
 {\rm  By using the change of variable $v = -u$, one gets that if $ f\in C\left( B(0,1)\setminus \{0\}\right)$ is  a  radial, bounded and  positive function and   $u\in C_{1, \alpha} $ is  a  bounded radial function satisfying
$$| \nabla u |^\alpha  { \cal M}^-(D^2 u) \leq -f r^{-\gamma}\quad \hbox{ in } B(0,1)\setminus \{0\}\, ,$$
  then, for $r$ sufficiently small, $u^\prime (r) \leq 0$, 
                         $\lim_{r\to 0} u^\prime (r)r^{{\tilde N_--1}}=0$ and
$$ 
u'(r)\leq    -\frac{\inf f}{\Lambda (({\tilde N}_--1)(1+\alpha) + 1-\gamma)} r^{1-\gamma\over 1+\alpha }\, .
$$
Moreover, if  $| \nabla u |^\alpha { \cal M}^-(D^2 u) = -f r^{-\gamma}\quad \hbox{ in } B(0,1)\setminus \{0\}$, then $|u^\prime (r) | \leq c r^{1-\gamma\over 1+\alpha }$ for a positive constant $c$. Hence  $u$ is locally Lipschitz continuous in $B(0,1)$ for $\gamma\leq 1$, it belongs to $C^1(B(0,1))$ if $\gamma <1$,  and it is locally H\"older continuous in $B(0,1)$ with exponent ${2+\alpha -\gamma\over 1+\alpha}$ for $\gamma >1$. 
 
Obviously, since 
                          ${\cal M}^-\leq { \cal M}^+$, one gets an analogous  conclusion when 
                          $$| \nabla u |^\alpha { \cal M}^+ ( D^2 u) \leq -f r^{-\gamma}.$$}
 \end{rema} 

 \begin{lemme}\label{fabiana}
Let $ f, g\in C\left( B(0,1)\setminus \{0\}\right)$ be  radial, bounded  functions such that $f> g$ and assume that  $u, v \in C_{1, \alpha} \cap {\cal C} ( \overline{B(0,1)})$ are    radial  functions  
 satisfying respectively 
\begin{equation}\label{subsol}
| \nabla u |^\alpha F ( D^2 u) \geq  f   r^{-\gamma} \ {\rm in } B(0,1)\setminus \{0\}
\end{equation}
 and 
 \begin{equation}\label{supersol}
 | \nabla v |^\alpha F ( D^2 v) \leq  g  r^{-\gamma}\qquad \hbox{ in } B(0,1)\setminus \{0\}\, .\end{equation}
 Then 
\begin{itemize}
 \item[(i)]  $u^\prime-v^\prime  \geq 0$ in a right neighborhood of $0$;
  \item[(ii)] In a right neighborhood of $0$ one has
  $$ |u^\prime |^\alpha u^\prime -|v^\prime |^\alpha v^\prime \geq  (1+\alpha)    \frac{\inf (f-g)}{\Lambda ( (N_--1)(1+\alpha) + 1-\gamma)} r^{1-\gamma}\, ;$$
 \end{itemize}
                        \end{lemme} 
                        \begin{proof}
                                                                                       
                                                Suppose  first that  $u^\prime -v^\prime$ has not a constant sign near zero. Then, there exists a decreasing sequence $\{r_n\}$ converging to 0, such that $u'(r_n)=0$ for all $n$, and  $u^\prime-v^\prime  \leq 0$ in $]r_{2n+1}, r_{2n}[$,    $u^\prime -v^\prime \geq 0$ in $]r_{2n+2}, r_{2n+1}[$. Since $(u^\prime-v^\prime ) (r_{2n}) = (u^\prime-v^\prime ) (r_{2n+1}) = 0$, there exists some $s_{2n}\in ]r_{2n+1}, r_{2n}[$ such that $(|u^{\prime }|^\alpha u^\prime -|v^\prime |^\alpha v^\prime )^\prime  ( s_{2n}) = 0$.  Then we use the inequality 
                                                 \begin{eqnarray}\label{G}
                                                   | \nabla u |^\alpha F( D^2 u)&-&| \nabla v |^\alpha  F ( D^2v)\nonumber\\
                                                   &=&
                                                  F( | \nabla u |^\alpha D^2 u)-F( | \nabla v |^\alpha D^2 v) \nonumber\\
                                                   &\leq&{\cal M}^+( | \nabla u |^\alpha D^2 u-| \nabla v |^\alpha D^2 v)\nonumber \\
                                                   & = & G_{\Lambda, \lambda} \left({( |u^\prime |^\alpha u^\prime -|v^\prime |^\alpha v^\prime ) (N-1)\over r} ,(1+\alpha)^{-1}( |u^\prime |^\alpha u^\prime -|v^\prime |^\alpha v^\prime )^\prime\right) 
                                                   \end{eqnarray}
This yields a  contradiction  since  here $G_{\Lambda, \lambda} \left({( |u^\prime |^\alpha u^\prime -|v^\prime |^\alpha v^\prime ) (N-1)\over r} ,(1+\alpha)^{-1}( |u^\prime |^\alpha u^\prime -|v^\prime |^\alpha v^\prime )^\prime\right) \leq 0$. 
                                                 Suppose now that $u^\prime -v^\prime \leq 0$ in a neighborhood of zero. 
                                                                       Using \eqref{G} above, 
                                                                                                                           $$\Lambda (|u^\prime |^\alpha u^\prime -|v^\prime |^\alpha v^\prime)^\prime  + \lambda  ( |u^\prime |^\alpha u^\prime -|v^\prime |^\alpha v^\prime ){ (N-1)\over r} \geq (f-g) r^{-\gamma}.$$
                                                     In particular 
                                                    $$ ( ( |u^\prime |^\alpha u^\prime -|v^\prime |^\alpha v^\prime ) r^{(1+\alpha)  (\tilde N+-1)} )^\prime \geq (1+\alpha)   (f-g) \Lambda^{-1}   r^{(1+\alpha)  (\tilde N_--1)-\gamma} $$
                                                    Since $f-g>0$, $( |u^\prime |^\alpha u^\prime -|v^\prime |^\alpha v^\prime ) r^{(1+\alpha)  (\tilde N+-1)} $ has a limit when $r$ goes to zero. We prove by contradiction that this limit is $\geq 0$.

                                                  In order to obtain this,   we  make an additional assumption which will be  deleted in a second step 
                                                    
                                                     {\bf First step} 
                                                    
                                                    Suppose that either  $\limsup _{r\rightarrow 0}u^\prime r^{ \tilde N+-1} \geq 0$, or   $\liminf _{r\rightarrow 0}v^\prime r^{ \tilde N+-1} \leq 0$. 
                                                     And assume by contradiction that
                                                    $$ ( |u^\prime |^\alpha u^\prime -|v^\prime |^\alpha v^\prime ) r^{(1+\alpha)  (\tilde N+-1)} \rightarrow -c<0.$$

                                                 - if $\limsup u^\prime  r^{\tilde N_+-1} \geq  0$, one derives that 
                                                      $ |v^\prime |^\alpha v^\prime \geq c  r^{( -N_++1) (1+\alpha) } $, and  then 
                                                      $ v^\prime \geq  c r^{(-N_++1)  } $ a contradiction with the  boundedness of $v$. 
                                                      
                                                      - If $\liminf v^\prime  r^{ \tilde N_+-1}\leq 0$, 
                                                      $$ |u^\prime |^\alpha u^\prime \leq -c  r^{( -N_++1) (1+\alpha) } $$ once more a contradiction with the boundedness of $u$. 
                                                      We have obtained that   $(|u^\prime |^\alpha u^\prime -|v^\prime |^\alpha v^\prime )r^{(1+\alpha)  (\tilde N+-1)} )\geq 0$, hence 
                                                      $u^\prime -v^\prime \geq 0$ and whatever is the sign of $(|u^\prime|^\alpha u^\prime )^\prime-(|v^\prime|^\alpha    v^\prime )^\prime$ one has 
                                                      
                                                       $$ ( ( |u^\prime |^\alpha u^\prime -|v^\prime |^\alpha v^\prime ) r^{(1+\alpha)  (\tilde N_--1)})^\prime \geq (1+\alpha ) (f-g) r^{ N_--1)(1+\alpha) -\gamma} \Lambda^{-1}$$
                                                       This implies by  integrating that 
                                                       
                                                       $$( |u^\prime|^\alpha u^\prime - |v^\prime |^\alpha v^\prime ) (r) \geq (1+\alpha) \Lambda^{-1}{ \inf( f-g) r^{1-\gamma} \over( N_--1)(1+\alpha)+1-\gamma}.$$
                                                       
                                                     {\bf Second step} 
                                                      
                                                      We use the previous step to compare $u$ with $w$ defined as follows  : 
                                                            For $0<\tau\leq\inf ( { \alpha+2-\gamma\over \alpha +1}, N_+)$, let us consider the radial function
 \begin{equation}\label{w}
 w(r)=L(1-r^\tau) +u(1)
 \end{equation} 
 where $L>0$ is a constant to be suitably chosen. 
 A  direct computation shows that
$$
\mathcal{M}^+(D^2w)\leq L \tau \Lambda ( |\tau-1| - (\tilde{N}_+-1) ) r^{\tau-2}\, .
$$
 so that
$$
\mathcal{M}^+(D^2w)\leq -CLr^{\tau-2}\qquad \hbox{ in } B(0,1)\setminus \{0\}
$$
with  $C=\tau \Lambda ( N_+-1-|\tau -1|)>0$.

 Then $$ | \nabla w |^\alpha \mathcal{M}^+ ( D^2 w) \leq -C(L\tau)^{ \alpha} L r^{\tau-2+ ( \tau-1)( \alpha)} \leq -|f|_\infty  r^{-\gamma} $$
  choosing conveniently $L$.  Note that  since $\gamma < \alpha+2 $, $\tau-1+ \tilde N_+-1>0$, and then $\lim w^\prime r^{ \tilde N_+-1} = 0$.  We can apply the first step and obtain that  for some positive constant $c$
   $$|u^\prime |^\alpha u^\prime (r) -|w^\prime|^\alpha w^\prime  (r) \geq  c r^{ 1-\gamma}$$
     this yields 
    $$| u^\prime|^\alpha u^\prime   \geq -\tau^{1+\alpha}  L^{1+\alpha} r^{( \tau-1)(1+\alpha) } + c r^{1-\gamma} $$
    In particular  using $0<\tilde N_+-2 < \tau-1+ \tilde N_+-1$
      one has 
    $ \limsup _{r\to 0} u^\prime (r) r^{ \tilde N_+-1} \geq  0$. 
    
    In the same manner if $v$ satisfies 
    $$ | \nabla v |^\alpha F( D^2 v) \leq g r^{-\gamma}$$
     Taking  $w = -L( 1-r^\tau) + v(1)$, as soon as $L$ is large enough,   $| \nabla w |^\alpha { \cal M}^-( D^2 w) \geq |g|_\infty r^{-\gamma}$. Using the first step  one gets $|v^\prime|^\alpha v^\prime  \leq |w^\prime|^\alpha w^\prime -c r^{1-\gamma}$, hence 
     $\liminf _{r\rightarrow 0}v^\prime r^{ \tilde N+-1} \leq 0$.  
    
As a consequence the result in the  first step extends to any $u$ and $v$ without  the  assumption  $ \limsup _{r\to 0} u^\prime (r) r^{ \tilde N_+-1} \geq  0$, or  $\liminf v^\prime  r^{ \tilde N_+-1}\leq 0$,  and items  (i) and (ii) are proved. 

\end{proof} 
 \begin{theo}\label{compa}
Let $f, g\in C\left(B(0,1)\right)$ be  radial functions and assume that $u, v\in C\left( \overline{B(0,1)}\right) \cap C_{1, \alpha} $,   are radial functions satisfying in $B(0,1)\setminus \{0\}$
 $$
 \begin{array}{c}                 
| \nabla u |^\alpha F( D^2 u) -\beta|u|^\alpha  u(r) r^{-\gamma} \geq f(r)r^{-\gamma} \\[2ex]
| \nabla v |^\alpha F( D^2 v) -\beta |v|^\alpha v(r) r^{-\gamma}  \leq g(r) r^{-\gamma}
\end{array}
$$
with $\beta\geq 0$ and   $f\geq g$ in $B(0,1)$, one of the inequalities being strict. Then  $u\leq v$ on $\partial B(0,1)$ implies $u\leq v$ in 
$\overline{B(0,1)}$.
 \end{theo}
              
\begin{proof} 
  We           
suppose by contradiction that 
$$\max_{\overline{B(0,1)}} (u-v) >0\, .
$$
If the maximum is achieved at 0, then $(u-v)(0) >0$ and, by the assumptions on $f$, $g$ and $\beta$, there exist $\delta>0$ and  a neighborhood on the right of $0$ on which 
               $$ \beta (|u|^\alpha u (r)-|v|^\alpha v (r) )r^{-\gamma}+(f(r)-g(r))r^{-\gamma} \geq \delta r^{-\gamma} >0$$
Using  Lemma \ref{fabiana},   one gets that  for some positive constant $c$, $|u^\prime|^\alpha  u^\prime -|v^\prime|^\alpha v^\prime  \geq  c  r^{1-\gamma}$, which contradicts the fact that   $u-v$ attains its maximum  at $0$. 
Hence, there exists $0<\bar r<1$ such that  $u(\bar r)-v(\bar r)=\max (u-v)$. Then for  $\epsilon >0$ there exists $\delta <\bar r$ that we can take small enough in order that $u( r) < v( r)+ \max ( u-v)-\epsilon$, for any  $r < \delta$.   Then  $w = v + \max ( u-v)$  satisfies 
$$ | \nabla w |^\alpha F( D^2 w)-\beta |w|^\alpha w r^{-\gamma} \leq  g(r)r^{-\gamma} $$
 Applying the comparison principle  in \cite{BD1} in $r> \delta$ for regular coefficients,  one gets that $u\leq w$ in $r> \delta$, and then $u-v\leq \sup (u-v)-\epsilon$ in the whole ball, a contradiction .                   \end{proof}
             
Next, we have the following existence, uniqueness and regularity  result.
                              
\begin{theo}\label{exi1}
              Let $f \in C (B(0,1))$ be a radial, bounded function. Let further $F$ be a rotationally invariant  operator satisfying \eqref{posh}, \eqref{FNL} . For $\beta \geq 0$ and $b\in \R$ 
   there exists a (unique when $\beta >0$ or when $f$ has a sign and $\beta = 0$) bounded radial function $u\in C(\overline {B(0,1)})\cap {\cal C}_{1, \alpha} $  satisfying
 \begin{equation}\label{dp}
 \left\{\begin{array}{cc}
            | \nabla u |^\alpha    F( D^2 u) -\beta |u|^\alpha { u r^{-\gamma}}  =  r^{-\gamma}   f(r) & \hbox{ in } \ B(0,1)\setminus \{0\}\\
                u (1)= b .& 
                \end{array}\right.
 \end{equation}               
Moreover, $u$ can be extended up to $\overline{B(0,1)}$, and one has: $u\in C^1(\overline{B(0,1)})$ if $\gamma<1$, $u$ is Lipschitz continuous in $\overline{B(0,1)}$ if $\gamma \leq 1$, $u$ is H\"older continuous in  $\overline{B(0,1)}$ with exponent ${2+\alpha -\gamma\over 1+\alpha} $ if $\gamma > 1$.            
              \end{theo}
 \begin{proof} 
 For every $n\in \N$ let us introduce the regularized Dirichlet boundary value problem
 \begin{equation}\label{dpapp}
 \left\{\begin{array}{cl}
| \nabla u_n |^\alpha  F( D^2 u_n) -\beta{ |u_n|^\alpha u_n (r^2+1/n)^{-\gamma/2}}  =  (r^2+1/n)^{-\gamma/2}   f(r) & \hbox{ in } \ B(0,1)\\
                u_n(1) = b &
                \end{array}\right.
 \end{equation}  
 which, by  \cite{BD1}, has a unique  radial solution $u_n\in C(\overline{B(0,1)})$. 
 For $0<\tau\leq 2-\gamma$, let us consider the radial function in \eqref{w}  replacing $u(1)$ by $b$ 
 where always $L>0$ is a constant  suitably chosen, and $\tau$ satisfies $\tau \leq \inf ( { \alpha+2-\gamma\over \alpha +1}, N_+)$. Then $$ | \nabla w |^\alpha \mathcal{M}^+ ( D^2 w) \leq -|f|_\infty  r^{-\gamma} \leq - (r^2+1/n)^{-\gamma/2}   f(r) .$$
   
  The standard comparison principle in \cite{BD1}  then implies that the sequence $\{ u_n\}$ is uniformly bounded in $C(\overline{B(0,1)})$.  Hence,  it is locally uniformly bounded in $C_{1, \alpha}$ , in particular uiformly Holder continuous, and, up to a subsequence, it is converging locally uniformly in $\overline{B(0,1)}\setminus \{0\}$ to a radial solution $u$ of problem \eqref{dp}, which is a globally  continuous  function belonging  to $C_{1, \alpha}$. 
  
  Let us now show that the constructed bounded radial solution $u$ satisfies the further regularity announced. We take now $\bar w$  a  solution of $$
\left\{
\begin{array}{cl}
| \nabla \bar w |^\alpha  {\cal M}^+ ( D^2 \bar w) = -(\|f\|_\infty+B+1) r^{-\gamma} & \quad \hbox{ in } B(0,1)\setminus \{0\}\\
 \bar w (1)= u(1) &
 \end{array}
 \right.$$
where $B>0$ is a constant such that $\beta |u|^{1+\alpha} \leq B$ in $B(0,1)$.  Lemma \ref{fabiana} and Remark \ref{lemmaformm} applied to $\bar w$ yield that 
                 $$ | \bar w^\prime| \leq c r^{1-\gamma\over 1+\alpha}$$
for some $c>0$.                  
 As a consequence, we have 
                   by Lemma \ref{fabiana} (i),   in a right neighborhood 
                   $(u-\bar w)^\prime(r) \geq   0\, .$
Hence, 
$$u^\prime(r) \geq -c  r^{1-\gamma\over 1+\alpha}$$
 for $r$ small enough. Analogously, we have
 $$
{\cal M}^- ( | \nabla u|^\alpha D^2 u+|\nabla \bar w |^\alpha  D^2 \bar w) \leq | \nabla u |^\alpha F( D^2 u) +| \nabla \bar w |^\alpha  { \cal M}^+( D^2 (\bar w) )\leq - r^{-\gamma}\qquad \hbox{ in } B(0,1)\setminus \{0\}$$
 which implies, by Remark \ref{lemmaformm}, 
  $$u^\prime(r) \leq -{\bar w}^\prime (r) \leq  c  r^{1-\gamma\over 1+\alpha }$$
 for $r$ small enough.  Arguing as in the proof of Lemma \ref{fabiana}, from the estimate $|u'(r)|\leq c r^{1-\gamma\over 1+\alpha }$ for $r$ sufficiently small, we deduce that $u$ is  Lipschitz continuous in $\overline{B(0,1)}$ if $\gamma \leq 1$,  it belongs to  $C^1(\overline{B(0,1)})$ if $\gamma <1$, and it is  H\"older continuous in $\overline{B(0,1)}$ with exponent ${2+\alpha -\gamma\over 1+\alpha}$ if $\gamma >1$.

The uniqueness is obvious when $\beta >0$. If $f$ has a sign  and $\beta = 0$ for example $f< 0$, then 
 if $u$ and $v$ are two solutions of 
$$ \left\{\begin{array}{lc} | \nabla u |^\alpha    F( D^2 u) =  r^{-\gamma}   f(r) & \hbox{ in } \ B(0,1)\setminus \{0\}\\
                u (1)= b&,
                \end{array}\right.$$
                 $u_\epsilon =( u-b) (1+\epsilon)+ b$ satisfies  $u_\epsilon (1) = b$, and 
                 $$ | \nabla u_\epsilon |^\alpha F( D^2 u_\epsilon)= f(r) (1+\epsilon)^{1+\alpha} r^{-\gamma} < {f(r) \over r^\gamma}$$
                  and then 
                  $u_\epsilon \geq v$. By  passing to the limit one obtains $u\geq v$. Exchanging the roles of $u$ and $v$ one gets the uniqueness. 

\end{proof}

The argument used in the above proof yields also the following  compactness result :                    
 \begin{theo}\label{compact}
 Let $\{u_n\}_n$ be a  uniformly bounded sequence of radial functions belonging to $C_{1, \alpha} $ and satisfying
              $$ | \nabla u_n|^\alpha F( D^2 u_n) = { f_nr^{-\gamma} }\qquad \hbox{ in } B(0,1)\setminus \{0\}\, ,$$ 
               where  the  $f_n$  are radial, bounded and  continuous in $B(0,1)$.
                If  $\{f_n\}_n$ is uniformly bounded, then $\{u_n\}_n$ is  equicontinuous, thus uniformly converging in $\overline{B(0,1)}$ up to a subsequence. 
                If $f_n$ is uniformly converging to $f\in C(\overline{B(0,1)})$, then,  up to a subsequence, $\{u_n\}_n$ is  uniformly converging to  a    radial solution $u\in C_{1, \alpha}\cap C(\overline{B(0,1)})$ of 
              $$| \nabla u |^\alpha F( D^2 u) = { r^{-\gamma} f }\qquad  \hbox{ in } B(0,1)\setminus \{0\}\, .$$
              \end{theo}

In the next results, we prove several properties of the \lq\lq smooth" eigenvalue defined in \eqref{lambdaprime}, that we recall here for the comfort of the reader
$$
\bar \lambda_\gamma' \, := \sup \{\mu\, : \  \exists\,   
        u \in {\cal C}_{1, \alpha}  ,\ u>0 \hbox{ in } B(0,1)\setminus \{0\}, \ u \hbox{ radial},   \ \ | \nabla u |^\alpha F (D^2 u) + \mu { u^{1+\alpha}  \over r^\gamma} \leq 0\}
 \,   . $$  
In Subsection 2.3 we will prove in fact that $\bar \lambda_\gamma'=\bar \lambda_\gamma$.

Let us start by proving  the validity of the maximum principle  below the value $\bar \lambda_\gamma'$.

\begin{theo}\label{maxpgamma}
 Let    $\mu < \bar \lambda_\gamma'$    and suppose that $u\in C(\overline{B(0,1)})\cap {\cal C}_{1, \alpha} $ is  a radial function satisfying 
 $$| \nabla u |^\alpha F(D^2 u) + \mu |u|^\alpha ur^{-\gamma} \geq 0 \qquad \hbox{ in } \ B(0,1)\setminus \{0\}\, .  $$
 If $u(1)\leq 0$,    then $u\leq 0$ in $\overline{B(0,1)}$.
                                   \end{theo}
                     
     \begin{proof}
     If $\mu < 0$, we just apply Theorem \ref{compa} with $v\equiv  g\equiv f \equiv 0$. So, we can assume without loss of generality that $\mu \geq 0$. 
      For $\mu^\prime \in ]\mu, \bar \lambda_\gamma'[$,  let $v\in{\cal C}_{1, \alpha} $ be a radial function satisfying
      $$| \nabla v |^\alpha F( D^2 v) + \mu^\prime  v^{1+\alpha}  r^{-\gamma}  \leq 0\, ,\quad v>0\qquad \hbox{ in } B(0,1)\setminus \{0\}\, .$$     We can assume without loss of generality that  $v>0$ on $\partial B(0,1)$, e.g. by performing a dilation in $r$  (this may change a little $\mu^\prime$ but we can still suppose by continuity that $\mu^\prime \in ]\mu, \bar \lambda_\gamma'[$).   
     Arguing as in the proof of Lemma \ref{fabiana}, since $| \nabla v |^\alpha { \cal M}^-(D^2 v) + \mu' v^{1+\alpha}  r^{-\gamma} \leq 0$,  we easily obtain that   $v^\prime$  has constant sign near zero. Assuming by contradiction that $v^\prime (r)\geq 0$ in a neighborhood of zero,  then $v$ is bounded in $B(0,1)\setminus \{0\}$. Hence,  Remark \ref{lemmaformm} applies and  yields $v'(r)\leq 0$ for $r$ small enough: a contradiction. This shows  that $v^\prime(r) \leq 0$ in a neighborhood of $0$.

 Then, there are two possible cases: either $\lim_{r\to 0}v(r)=+\infty$ or $v$ can be extended as a continuous function on $\overline{B(0,1)}$. In the first case, let $\epsilon >0$ be given and  let $\delta >0$ be small  in order that 
 $\sup u  \leq \epsilon \inf_{r< \delta} v$, then  applying the  comparison principle in \cite{BD1} for regular coefficients in the domain $B(0,1)\setminus\overline{B(0, \delta)}$, one gets $u\leq \epsilon v$ in $r> \delta$, hence $u\leq \epsilon v$ everywhere.  Letting $\epsilon\to 0$, we get the conclusion. On the other hand, if $v$ is bounded and continuous on $\overline{B(0,1)}$, we can argue by contradiction : Let us assume that $u$ is positive somewhere in $B(0,1)$, so that   ${ u\over v}$  has a positive maximum on $\overline{B(0,1)}$, achieved at some point inside $B(0,1)$. Up to a multiplicative constant for $v$, 
  we  can suppose that 
  $$\max_{\overline{B(0,1)}} {u\over v}=1\, ,$$  
so that  $ u(r)\leq v(r)$. If the maximum is achieved at $0$, then one has 
       $u(0)= v(0)>0\, . $
  By continuity, for   $r$  small enough one has 
       $$| \nabla u |^\alpha  F(D^2u)-| \nabla v |^\alpha  F(D^2v)\geq {1\over 2} ( \mu^\prime v^{1+\alpha} (0)-\mu u^{1+\alpha} (0)) r^{-\gamma}\, .$$
Since $\mu^\prime v^{1+\alpha}(0)-\mu u^{1+\alpha} (0)= (\mu'-\mu)v^{1+\alpha} (0) >0$, we can  use Lemma \ref{fabiana} (ii) and  we get that  
    $(u -v)' (r) >0$ in a right neighborhood of $0$. This is a contradiction to  $u-v$ has a maximum point at zero.
 
 Hence, we have that   
  $1= \max {u\over v} > \frac{u(0)}{v(0)}$. Let us select $\eta  <1$ such that 
     $\eta  >\max\{ \left( {\mu\over \mu^\prime}\right)^{1\over 1+\alpha} , {u(0)\over v(0)}\}$.  Then,  the function $u-\eta v$ has a  positive maximum achieved at some point  $0<\bar r <1$.  
      Since $\bar r$ is an interior point, using the arguments in the  comparison Theorem for regular coefficients in \cite{BD1}, or using remark \ref{maxu-v} below, valid for  functions depending on one variable, 
     one has 
            $$ | \nabla u |^\alpha F( D^2u)( \bar r) \leq  | \nabla ( \eta v) |^\alpha F( D^2\eta  v)( \bar r)$$
        and then  using \eqref{posh}, we get
      $$ -\mu  v^{1+\alpha} ( \bar r){\bar r}^{-\gamma}\leq -\mu  u^{1+\alpha} ( \bar r) {\bar r}^{-\gamma} \leq | \nabla u |^\alpha F( D^2 u(\bar r))\leq | \nabla \eta v |^\alpha F( \eta D^2  v (\bar r)) \leq  -\mu^\prime  \eta^{1+\alpha}  \, v^{1+\alpha} ( \bar r) {\bar  r}^{-\gamma}\, ,$$
which gives the  contradiction $\mu\geq \mu' \eta^{1+\alpha} $. 
       
        \end{proof}
     \begin{rema}\label{maxu-v}
     Suppose that $\alpha >-1$, that $u$ and $v$ are  defined  $]\bar r-h, \bar r+ h[$ are ${ \cal C}^1$ and such that $|u^\prime |^\alpha u^\prime , |v^\prime |^\alpha v^\prime \in { \cal C}^1$. Then  if $\bar r$ is a maximum point, one has 
     $$ u^\prime (\bar r) = v^\prime ( \bar r) \ {\rm and } \ (|u^\prime |^\alpha u^\prime )^\prime (\bar r) \leq   (|v^\prime |^\alpha v^\prime )^\prime (\bar r).$$
     Indeed, the result is obvious when $u^\prime ( \bar r) \neq 0$, since in that case the assumption   $u^\prime |^\alpha u^\prime , |v^\prime |^\alpha v^\prime \in { \cal C}^1$ implies that $u$ and $v$ are ${ \cal C}^2$ on $\bar r$. Then 
     $$  ( |u^\prime |^\alpha u^\prime )^\prime (\bar r) = (1+\alpha) |u^\prime |^\alpha (\bar r)) u^{ \prime \prime} ( \bar r)\leq (1+\alpha) |u^\prime |(\bar r))v^{ \prime \prime}( \bar r) = (1+\alpha) |v^\prime |^\alpha (\bar r)) v^{ \prime \prime} ( \bar r) = 
     (|v^\prime |^\alpha v^\prime )^\prime ( \bar r)$$
      If now $u^\prime ( \bar r) = v^\prime ( \bar r)=0$, let us consider $u_\epsilon = u-\epsilon r$, $v_\epsilon = v-\epsilon r$, $u_\epsilon-v_\epsilon$ has  its maximum also achieved on $\bar r$ and $u_\epsilon^\prime ( \bar r) \neq 0$, hence by the first part 
      $$ (|u_\epsilon^\prime |^\alpha u_\epsilon ^\prime )^\prime (\bar r) \leq  (|v_\epsilon ^\prime |^\alpha v_\epsilon ^\prime)^\prime(\bar r) $$
       and passing to the limit  one gets 
       $$ ( |u^\prime |^\alpha u^\prime )^\prime (\bar r) \leq  (|v^\prime |^\alpha v^\prime )^\prime(\bar r)$$
       \end{rema}
The next result provides the existence, uniqueness  and regularity of solutions   below the eigenvalue $\bar \lambda_\gamma'$.

\begin{theo}\label{exilambda}
 Let   $f$ be a  radial and continuous function in $\overline{B(0,1)}$ satisfying $f\leq 0$, with $f$ not identically zero.   Then, for every $\mu < \bar \lambda_\gamma'$ there exists a unique, positive,  bounded, radial function $u\in {\cal C}_{1, \alpha}  $ satisfying
 $$
 \left\{
 \begin{array}{cl} 
| \nabla u |^\alpha   F( D^2 u)+ \mu u^{1+\alpha}   r^{-\gamma}  =  f (r) r^{-\gamma} & \hbox{ in } B(0,1)\setminus \{0\}\\
   u(1)=0 .&
 \end{array}\right.  $$
Moreover, $u$ can be extended as a strictly positive continuous function in $B(0,1)$,   Lipschitz continuous in $\overline{B(0,1)}$ if $\gamma \leq 1$, ${\cal C}^1$ if $\gamma <1$, $({2+\alpha-\gamma\over 1+\alpha} )-$H\"older continuous in $\overline{B(0,1)}$ if $\gamma >1$. 
                \end{theo}
 \begin{proof} 
 
 As in the proof of Theorem \ref{maxpgamma}, we can assume without loss of generality that $\mu>0$, otherwise the conclusion just follows from Theorem \ref{exi1}. The uniqueness of $u$ follows from Theorem \ref{maxpgamma}.
 
  As far as existence is concerned,  let us recursively define a sequence $\{u_n\}_{n\geq 0}$ as follows: we set
$u_0 \equiv  0\, ,$
and  by using Theorem \ref{exi1}, we  define $u_{n+1}\in  {\cal C}_{1, \alpha }  \cap C(\overline{B(0,1)})$ as the unique bounded radial solution of
$$
\left\{
\begin{array}{cc}
| \nabla u_{n+1}|^\alpha  F( D^2 u_{n+1})  = ( f -\mu u_n^{1+\alpha}  )  r^{-\gamma}  & \hbox{ in } B(0,1)\setminus \{0\}\\
 u_{n+1}(1)=0. &
 \end{array}\right.$$
  By Theorem \ref{compa}, we have that  $u_{n+1} \geq 0$, hence it is strictly positive in $B(0,1)\setminus \{0\}$  by the standard strong maximum principle, since $f$ is not identically zero. In particular,  $u_n$ is not identically zero for all $n\geq 1$. By applying  the comparison principle in Theorem \ref{compa} again,  we deduce also that   $u_{n+1} \geq u_n$. Let us prove that $\{u_n\}_n$ is uniformly bounded. If not, by setting
                $v_n = {\|u_n\|_\infty}^{-1} u_n $ and $k_n =  \|u_{n+1}\|_\infty^{-1}\|u_n\|_\infty\leq 1$, one gets that $v_{n+1}$ satisfies
$$
\left\{
\begin{array}{cc}
| \nabla v_{n+1}|^\alpha  F ( D^2 v_{n+1}) = \left( \frac{f(r)}{\|u_{n+1}\|^{1+\alpha} _\infty}  -\mu k_n^{1+\alpha}  v_n^{1+\alpha} (r)\right)  r^{-\gamma}& \hbox{ in } B(0,1)\setminus \{0\}\\
 v_{n+1}(1)=0 & 
 \end{array}\right.
  $$
                  Since $\{v_n\}_n$ is uniformly bounded, by applying Theorem \ref{compact},  we can extract  a subsequence still denoted by $\{v_n\}_n$ uniformly converging to  a function $v\geq 0$ satisfying
 $$
 \left\{
\begin{array}{cc}
| \nabla v |^\alpha F ( D^2 v)+ \mu k^{1+\alpha}  v^{1+\alpha}  r^{-\gamma} =0 & \hbox{ in } B(0,1)\setminus \{0\}\\
 v(1)=0 & 
 \end{array}\right.
$$
 where $k \leq 1$ is the limit  of some converging subsequence of  $\{k_n\}_n$.
Since $v$ is a radial solution, one has that  $v\in {\cal C}_{1, \alpha} $  and, since $\mu k \leq \mu < \bar \lambda_\gamma'$,  Theorem \ref{maxpgamma} yields $v\leq 0$. Hence, we get $v\equiv 0$, a contradiction with $\|v\|_\infty = 1$. 
                 
                  We have obtained that $\{u_n\}_n$ is bounded, and using once more  Theorem \ref{compact}, we deduce that $\{u_n\}$ uniformly converges to some $u$, which  satisfies the desired equation. 
            By the strong maximum principle, we get that $u>0$ in $B(0,1)\setminus\{0\}$. Moreover, by Remark \ref{lemmaformm}, we have that     $u^\prime (r)\leq 0$ for $r>0$ small enough, which implies  $u(0)>0$.  Finally,    the global regularity of $u$ follows from Remark \ref{remC1alpha} and Theorem \ref{exi1}. In fact $u^\prime <0$ and then $u\in { \cal C}^2( B(0,1)\setminus\{0\})$.  
            
             We now prove the uniqueness. Suppose that $u$ and $v$ are two ( positive )  radial solutions of 
             $$| \nabla u |^\alpha F( D^2 u) + \mu u ^{1+\alpha} r^{-\gamma} = f r^{-\gamma}, \ u(1)=0$$
             If $f\equiv 0$, the maximum  principle implies that $u\equiv 0$, so we can suppose that $f$ is not identically $0$. 
             
              We prove that $u\leq v$, next, exchanging the roles of $u$ and $v$ one gets the uniqueness. Suppose by contradiction that 
              $\sup { u \over v} = \eta >1$. (The supremum is well defined by Hopf principle in \cite{BD1}). 
               Then $u\leq \eta v$, and the supremum is achieved either  on $\bar r >0$, or "on the boundary" which means that 
               $\lim_{r\rightarrow 1} { u \over v} = \eta$, or on $0$. Suppose that we are in one of the first situations, then since $u^\prime <0$ by the equation and by lemma \ref{fabiana} one gets using the strong comparison principle, proposition 4.4  in \cite{BD2}, that 
               $u= \eta v$ everywhere in $\overline{B(0,1)}\setminus\{0\}$ and then by continuity $u = \eta v$ in $\overline{B(0,1)}$. This contradicts the equations since one would obtain then 
               $f = \eta f$, a contradiction. 
               We are back to the case where $u< \eta v$ for any $r\neq 0$. Then one has 
               $$ { \cal M}^+ ( | \nabla u |^\alpha D^2 u-| \nabla (\eta v) |^\alpha D^2 (\eta v)) \geq( (1-\eta) f + \mu ( \eta v-u)) r^{-\gamma} >0$$
                and then using remark \ref{strict} one gets that 
                $u^\prime > \eta v^\prime
$ in a neighborhood of $0$ , which contradicts the fact that $u-\eta v$ has a supremum at zero. 
                    
\end{proof} 

We can now prove  that the smooth eigenvalue $\bar \lambda_\gamma'$ is actually achieved on smooth eigenfunctions.
\begin{theo}\label{exieig3}
                     There exists $u\in C(\overline{B(0,1)})\cap C^2(B(0,1)\setminus \{0\})$,   radial,  strictly positive  in $B(0,1)$  and satisfying
$$
 \left\{
\begin{array}{cl}
 | \nabla u |^\alpha F( D^2 u) + \bar \lambda_\gamma' u^{1+\alpha}    r^{-\gamma} = 0& \hbox{ in } B(0,1)\setminus \{0\}\\
 u(1)=0 & 
 \end{array}\right.
 $$
Furthermore, in $\overline{B(0,1)}$, $u$ is   Lipschitz continuous  when $\gamma \leq 1$, ${ \cal C}^1$ when $\gamma <1$,  and  H\"older continuous  with exponent ${2+\alpha-\gamma\over 1+\alpha} $ if $\gamma >1$.  
\end{theo}
\begin{proof} 
We consider a sequence $\{\lambda_n\}$, with  $\lambda_n \rightarrow \bar \lambda_\gamma'$ and  $\lambda_n <\bar 
\lambda_\gamma'$ and, for all $n$, the solution
 $u_n\in C(\overline{B(0,1)})\cap {\cal C}_{1, \alpha} $, ( even $ C^2(B(0,1)\setminus \{0\})$ by remark. \ref{remC1alpha})   provided by Theorem \ref{exilambda}
                          of 
                           $$ | \nabla u_n |^\alpha F( D^2 u_n) + \lambda_n u_n^{1+\alpha}   r^{-\gamma} = {- r^{-\gamma} },\quad  u_n(1)=0 \ .$$
We claim that the positive sequence $\{ \|u_n\|_\infty\}_n$ is unbounded.     Indeed, arguing by contradiction,  if $\{ u_n\}_n$ is uniformly bounded, then,  by using Theorem \ref{compact} and considering a subsequence if necessary,                            we obtain  that there exists a  solution $u\in C(\overline{B(0,1)})\cap {\cal C}_{1, \alpha} $, $u\geq 0$, of
                            $$ | \nabla u |^\alpha F( D^2 u) + \bar \lambda_\gamma' u ^{1+\alpha}  r^{-\gamma}  = {- r^{-\gamma} },\quad  u(1)=0\, . $$
 Then, arguing as in the proof of Theorem \ref{exilambda}, we deduce that $u$ is strictly positive in $B(0,1)$, and by taking $0< \epsilon < {1\over \|u\|_\infty^{1+\alpha}}$,  we see that $u$ satisfies  
                            $$ | \nabla u |^\alpha F( D^2 u) + (\bar \lambda_\gamma'  +\epsilon)  u ^{1+\alpha}   r^{-\gamma}     \leq 0, $$
                             a contradiction to the definition of $\bar \lambda_\gamma'$.  It then follows that the sequence $\{u_n\}_n$ is not uniformly bounded.  Normalizing and considering  a subsequence, letting $n\to \infty$ yields the existence of  $u\in  
 C(\overline{B(0,1)})\cap {\cal C}_{1, \alpha}$ satisfying 
$$ | \nabla u |^\alpha F( D^2 u) + \bar \lambda_\gamma' u^{1+\alpha}   r^{-\gamma} = 0,  \ {\rm in} \   B(0,1) \setminus \{0\} \      , u(1)=0 \ .$$
Finally,  the strict positivity of $u$ in $B(0,1)$ and its global regularity in $\overline{B(0,1)}$ follow by arguing as in the proof of Theorem \ref{exilambda}.

                             \end{proof}

\begin{prop}\label{simple}
The first eigenvalue is simple. 
\end{prop}

 \begin{proof} 
 
  We have proved that there exists a positive bounded eigenfunction $v$. Take another eigenfunction $u>0$.   Let us define 
  $\eta^{-1} = \sup { v \over u}$ which is well defined by Hopf principle and is finite. Then $u-\eta v$ is nonnegative in $\overline{B(0,1)}\setminus\{0]$ and satisfies 
  \begin{eqnarray*}
  {\cal M}^-( | \nabla u |^\alpha  D^2 u -| \nabla (\eta v) |^\alpha D^2(\eta  v)) 
  &\leq& F( | \nabla u|^\alpha D^2 u)- F( | \nabla(\eta  v) |^\alpha D^2 (\eta v)) \\
  &=& -\bar \lambda u^{1+\alpha} +\bar \lambda  ( \eta v)^{1+\alpha} 
  \\
  &\leq & 0
  \end{eqnarray*}
  If by contradiction $u-\eta v >0$ on some point in $B(0,1)\setminus\{0\}$,   using the fact that $u^\prime $ and $v^\prime$ are $<0$ and using the strong comparison principle, Proposition 4.4  in \cite{BD2} one gets $u-\eta v >0$ in the whole of $B(0,1)\setminus\{0\}$. 
   We now distinguish two   cases
   
    -If  $u(0) = +\infty$, then necessarily the supremum is "achieved " on $1$ say 
   $$ {1\over \eta} = \lim_{r\rightarrow 1} { v \over u} = { v^\prime (1)\over u^\prime (1)}$$
   But this contradicts Hopf lemma in \cite{BD1}, \cite{BD2}. 
   
   - If $u(0) < \infty$, the maximum could be achieved on $0$ , but then we get from  Lemma \ref{fabiana} that $(u-\eta v)^\prime \leq 0$ for $r$ sufficiently small, hence $u-\eta v$ cannot have a strict minimum at zero. 
    Thus we have obtained that all the eigenfunctions are bounded and multiple to each others. 
    \end{proof}
\begin{prop}\label{compasous}
Suppose that $0\leq \mu < \mu^\prime \leq \bar \lambda$.  Suppose that $u$  and $v>0$ are continuous, radial, ${\cal C}_{1, \alpha}$  satisfy respectively 
 $$ | \nabla u |^\alpha F( D^2 u) + \mu  |u|^{\alpha} ur^{-\gamma} \geq g$$
 $$ | \nabla v |^\alpha F( D^2 v) + \mu^\prime v^{1+\alpha} r^{-\gamma}\leq  f\leq 0$$
 with $g\geq f$ 
 and $u\leq v$  on the boundary. Then $u\leq v$ in $B(0,1)$. 
  \end{prop} 
  \begin{proof}
    One can assume that $v>0$ on the boundary. Indeed, let $\lambda <1$  close to $1$. Then $v_\lambda (x) = v( \lambda x)$, $u_\lambda (x) = u( \lambda x)$ satisfy 
    $$ | \nabla u_\lambda |^\alpha F( D^2 u_\lambda) + \mu  \lambda ^{1+\alpha-\gamma} |u_\lambda|^{\alpha} u_\lambda r^{-\gamma}  \geq \lambda^{1+\alpha} g_\lambda$$
 $$ | \nabla v_\lambda |^\alpha F( D^2 v_\lambda) + \mu^\prime v_\lambda^{1+\alpha}  \lambda ^{1+\alpha-\gamma}r^{-\gamma} \leq \lambda^{1+\alpha}  f_\lambda \leq 0$$
  we have $u_\lambda \leq v_\lambda$ on  $r=1$, and $v_\lambda(1) >0$.  If the result is true in the case $v(1)>0$, letting $\lambda$ tend to $1$ yields the desired result.  So we suppose that $v(1)>0$. Let 
  $\eta = \sup { u \over v}$, and suppose that $\eta >1$.  One has then 
  $\eta v-u\geq 0$ and 
  $$ | \nabla u |^\alpha F( D^2 u) \geq( g-\mu u^{1+\alpha} )r^{-\gamma} \geq (g-\mu \eta ^{1+\alpha} v^{1+\alpha}) r^{-\gamma} \geq (f-\mu \eta ^{1+\alpha} v^{1+\alpha}) r^{-\gamma}  \geq | \nabla ( \eta v) |^\alpha F( D^2( \eta v))$$
   Using once more Proposition 4.4 in \cite {BD2} one obtains that  for any $r\neq 0$, either $u<\eta v$ or 
   $u= \eta v$. Using the same arguments as in the proof of \ref{simple} one gets that $u \equiv  \eta v$ everywhere. This contradicts the inequations since 
   $$ | \nabla \eta v |^\alpha F( D^2 \eta v) + \mu ( \eta v)^{1+\alpha} r^{-\gamma} =  | \nabla u |^\alpha F( D^2 u) + \mu  |u|^{\alpha} ur^{-\gamma}  \geq g \geq f \geq | \nabla \eta v |^\alpha F( D^2 \eta v) + \mu^\prime ( \eta v)^{1+\alpha} r^{-\gamma}$$
    which implies $\mu = \mu^\prime$, a contradiction. 
   \end{proof}

\subsection{The first eigenvalue inherited from $\R^+$}
In all that section we assume that $F= { \cal M}^+$.  The case of ${ \cal M}^-$ is easily derived by  exchanging $\lambda$ and $\Lambda$. We present an alternative proof  of the existence of  radial eigenfunctions related to the eigenvalue $\bar \lambda_\gamma$. Here, the idea is to show   the existence of global solutions defined in  $(0,+\infty)$ of the ODE associated with  radial solutions of  the equation  
       $$ {\cal M}^+ ( D^2 u) = - {u \over r^\gamma}\, \qquad \hbox{ in } \R^N\setminus \{0\}\, .$$  
  For $\gamma <1$ let us define 
   $$V_{r_0} = \{ u \in C ( [0, r_0])\, : \  |u(r)-1| \leq {1\over 2}, u(0) = 1\}. $$  
   where $r_o$ will be chosen later , and  let $T$  be defined on $V(r_o)$ by 
  
   $$Tu(r) = 1-\int_0^r\left(  {1+\alpha \over \lambda}\right)^{1\over 1+\alpha} s^{1-N} (\int_0^s u^{1+\alpha }  t^{(N-1) (1+\alpha)-\gamma} dt )^{1\over 1+\alpha}ds$$
   
 We prove below that   for $r_o$ small enough $T$ is a contraction mapping and then it possesses a unique fixed point. 
     Let    $r_o$  be defined by 
       \begin{equation}\label{ro1}
         r_o^{2+\alpha-\gamma\over 1+\alpha} < (2+\alpha-\gamma)  3^{-1-|\alpha| } (1+\alpha^+)^{-2-\alpha \over 1+\alpha}   \lambda^{1\over 1+\alpha} ( (N-1)(1+\alpha) + (1-\gamma))^{1\over 1+\alpha}. 
         \end{equation}  
         Then   $T$ maps $V_{r_o}$ into itself for $r_o$ small enough.  Indeed 
            $$ |Tu(r)-1|\leq\left( {3\over 2}\right)^{1+\alpha}  r^{2+\alpha-\gamma\over 1+\alpha} {(1+\alpha)^{2+\alpha\over 1+\alpha} \over \lambda^{1\over 1+\alpha} (2+\alpha-\gamma)(( N-1)(1+\alpha) + (1-\gamma))^{1\over 1+\alpha}}$$
    To prove that $T$ is a contraction mapping on $V( r_o)$,        we consider separately the cases $\alpha >0$ and $\alpha <0$. 
      
     - If $\alpha >0$,  

         Let us define, when $\alpha >0$  the space 
         ${\bf L}^\alpha(]0,r[)  = L^{ \alpha+1} ( ]0,r[, t^{(N-1)(1+\alpha)-\gamma})$ the space of functions with power $\alpha+1$ integrable  for the measure $ t^{(N-1)(1+\alpha)-\gamma}dt$. 
    Using Minkowski inequality when $\alpha >0$ one has 
    $$ ||u|_{{\bf L}^\alpha(]0,r[ } -|v|_{ {\bf L}^\alpha(]0,r[)} | \leq |u-v|_{ {\bf L}^\alpha(]0,r[) } $$
     and then 
     \begin{eqnarray*}
      |Tu -Tv|(r) &\leq& (1+\alpha) ^{1\over 1+\alpha} \int_0^r {1\over \lambda^{1\over 1+\alpha} s^{N-1}}|u-v|_{ {\bf L}^\alpha(]0,s[)}ds \\
      & \leq& {(1+\alpha)^{1\over 1+\alpha}  \over \lambda^{1\over 1+\alpha} ( N-1)(1+\alpha) + (1-\gamma))^{1\over 1+\alpha} }|u-v|_\infty  \int_0^r t^{1-N} t^{N-1+ {1-\gamma\over 1+\alpha}} dt\\
      &\leq & |u-v|_\infty r^{2+\alpha-\gamma\over 1+\alpha} {(1+\alpha)^{2+\alpha\over 1+\alpha}  \over (2+\alpha-\gamma) \lambda^{1\over 1+\alpha} (( N-1)(1+\alpha) + (1-\gamma))^{1\over 1+\alpha}}\\
      &\leq & {1\over 2} |u-v|_\infty.
      \end{eqnarray*}
      So $T$ is a contraction mapping on $V_{r_o}$ under the assumption \eqref{ro1}.

      If $\alpha <0$ 
          We prove that under the assumption  \eqref{ro1}, $T$ is a contraction mapping . We use twice  the mean value's Theorem  : 
      
      \begin{eqnarray*}
     &&  |Tu(r)-Tv(r)| \leq {(1+\alpha)^{1\over 1+\alpha} \over \lambda^{1\over 1+\alpha} }  \int_0^r {1\over s^{N-1} }\left\vert  \left(( \int_0^s u^{1+\alpha} t^{(N-1)(1+\alpha)-\gamma}dt)^{1\over 1+\alpha} -( \int_0^s  v^{1+\alpha}  t^{(N-1)(1+\alpha)-\gamma}dt)^{1\over 1+\alpha}\right) \right\vert ds\\
       &\leq & {(1+\alpha)^{1\over 1+\alpha}\over 1+\alpha}((N-1)(1+\alpha) + (1-\gamma))^{\alpha \over 1+\alpha}  \left({3\over 2}\right)^{-\alpha}\times \\
              && \int_0^r  s^{1-N} s^{-\alpha (N-1)-{(1-\gamma)\alpha\over 1+\alpha}} \left\vert \int_0^s ( u^{1+\alpha}  -v^{1+\alpha} ) (t)t^{(N-1)(1+\alpha)-\gamma } dt \right\vert ds \\
       &\leq &(1+\alpha) ^{1\over 1+\alpha} 2^{-\alpha} ((N-1)(1+\alpha) + (1-\gamma))^{\alpha  \over 1+\alpha} \left({3\over 2}\right)^{-\alpha} \int_0^r  s^{-(1+\alpha) (N-1)-{(1-\gamma)\alpha\over 1+\alpha}} \int_0^s  |u-v|_\infty  t^{(N-1)(1+\alpha)-\gamma} dt ds\\
       &\leq &  3^{-\alpha} (1+\alpha) ^{1\over 1+\alpha} ((N-1)(1+\alpha) + (1-\gamma))^{-1 \over 1+\alpha}  |u-v|_\infty \int_0^r  s^{-(1+\alpha) (N-1)-{(1-\gamma)\alpha\over 1+\alpha}} s^{(N-1)(1+\alpha)+ 1-\gamma}  ds \\
       &\leq & 3^{-\alpha} {(1+\alpha)^{2+\alpha\over 1+\alpha}  \over 2+\alpha-\gamma} ((N-1)(1+\alpha) + (1-\gamma))^{-1\over 1+\alpha}  |u-v|_\infty r^{2+\alpha-\gamma\over 1+ \alpha}  \\
        &\leq & {1\over 2} |u-v|_\infty.
              \end{eqnarray*} 
               
         by the choice of $r_o$.

We need to prove that the fixed  point above is a solution of $$| \nabla u |^\alpha { \cal M}^+ ( D^2 u) =-u^{1+\alpha}  r^{-\gamma}$$
 for that aim we prove that if $\gamma <1$, $u^\prime <0$ and $(|u^\prime |^\alpha u^\prime )^\prime <0$. 
  The first assertion is immediate, and then since $u(0)  = 1$,  $u\leq 1$,  one gets that
  $$|u^\prime |^\alpha u^\prime \geq  -{1+\alpha\over\lambda  r^{(N-1)(1+\alpha)}} \int_0^r   t^{(N-1)(1+\alpha)-\gamma} dt \geq -{1+\alpha\over \lambda }{ r^{1-\gamma} \over (N-1)(1+\alpha)+1-\gamma}$$
  
 From this we derive 
   $$u(r) \geq 1-{((1+\alpha)\lambda^{-1})^{1\over 1+\alpha}  \over 2 +\alpha-\gamma} {(1+\alpha) r^{2+\alpha-\gamma \over 1+\alpha} \over ( (N-1)(1+\alpha)+1-\gamma)^{1\over 1+\alpha}}.$$
   
  Then 
  $$ |u|^\alpha u (r) \geq 1-{(1+\alpha)^{3+2\alpha\over 1+\alpha} \lambda^{-1\over 1+\alpha} \over 2+\alpha-\gamma} {r^{2+\alpha-\gamma \over 1+\alpha} \over ( (N-1)(1+\alpha)+1-\gamma)^{1\over 1+\alpha}}+o( r^{2+\alpha-\gamma  \over 1+\alpha} ).$$

Using the fact that $u$ is a fixed point of $T$   \begin{eqnarray*}
   (|u^\prime |^\alpha u^\prime)^\prime &=& -(1+\alpha) u^{1+\alpha} \lambda^{-1} r^{-\gamma} -(N-1) r^{-1} (1+\alpha) |u^\prime |^\alpha u^\prime\\
   &\leq &  -r^{-\gamma} (1+\alpha)\lambda^{-1}
   +(N-1) { (1+\alpha) ^2\over \lambda ((N-1)(1+\alpha)+1-\gamma)}  r^{-\gamma}+O( r^{(1-\gamma )(2+\alpha) \over 1+\alpha} ) \\
     &=& r^{-\gamma} (1+\alpha) \lambda^{-1} {-1+\gamma\over  (N-1)(1+\alpha) + 1-\gamma} + o(r^{-\gamma})\\
   &<&0
   \end{eqnarray*}

  We now consider the case $\gamma \geq 1$. Here we define 
    \begin{equation}\label{r1}
            r_o^{2+\alpha-\gamma\over 1+\alpha} < (2+\alpha-\gamma)2^{\alpha^+}   3^{-1-|\alpha| } (1+\alpha^+)^{-2-\alpha \over 1+\alpha}   \Lambda^{1\over 1+\alpha} ( (\tilde N_+-1)(1+\alpha) + (1-\gamma))^{1\over 1+\alpha}
         \end{equation}  and as  in the case $\gamma <1$
          $V_{r_0} = \{ u \in C ( [0, r_0])\, : \  |u(r)-1| \leq {1\over 2}, u(0) = 1\}. $.  We define $T$ n
          by 
  $$Tu(r) = 1-\int_0^r \left( {1+\alpha \over \Lambda s^{ (\tilde N_+-1)(1+\alpha)} } \int_0^s u^{1+\alpha}  (t) t^{(\tilde N_+-1)(1+\alpha)-\gamma} dt \right)^{1\over 1+\alpha} ds$$
  Following the same calculations as in the case $\gamma <1$,  $T$ maps $V_{r_0} $ into itself and is a contraction mapping .   We need to prove that 
  $|u^\prime |^\alpha u^\prime$ has a derivative $\geq 0$. 
  
  If $\gamma >1$ we have  for some positive constant $c$ which can vary from one line to another,  $|u^\prime |^\alpha u^\prime \geq - c r^{1-\gamma \over 1+\alpha}$ and then 
  $$ u(r) \geq 1-c r^{2+\alpha-\gamma \over 1+\alpha}$$
   which  in turn implies that 
   $$ |u^\prime |^\alpha u^\prime \leq -{1\over \Lambda} { r^{1-\gamma} \over ((\tilde N_+-1)(1+\alpha) + 1-\gamma} +o(  r^{1-\gamma} )$$ 
    and then  using once more $u<1$, 
    \begin{eqnarray*}
   ( |u^\prime |^\alpha u^\prime )^\prime& \geq & -(1+\alpha) r^{-\gamma} \Lambda^{-1}  + (1+\alpha) { (\tilde N_+-1)(1+\alpha) \over r}  {1\over \Lambda} { r^{1-\gamma} \over (\tilde N_+-1)(1+\alpha) + 1-\gamma} +o(  r^{1-\gamma} )\\
    & =& { (1+\alpha) \Lambda^{-1}(\gamma-1)\over   (\tilde N_+-1)(1+\alpha) + 1-\gamma} r^{-\gamma} +o(  r^{1-\gamma} )
    \end{eqnarray*}
    
  There remains to do the case $\gamma = 1$.   Here we need to go further in the second term of the DL : 
  
   $$|u^\prime |^\alpha u^\prime = -{(1+\alpha) \over \Lambda} r^{ -(\tilde N_+-1)(1+\alpha)} \int_0^r u^{1+\alpha}  (t) t^{(\tilde N_+-1)(1+\alpha)-1} dt \geq -{1 \over \Lambda (\tilde N_+-1)} $$
   Then 
   $$ u(r) \geq 1- \left( {1\over \Lambda (\tilde N_+-1)} \right)^{1\over 1+\alpha} r$$
    hence 
    $$ |u|^\alpha u \geq 1- (1+\alpha)\left( {1\over \Lambda (\tilde N_+-1)} \right)^{1\over 1+\alpha} r + O(r^2)$$
     Hence 
     $$ |u^\prime |^\alpha u^\prime \leq -{1\over \Lambda ( \tilde N_+-1)} + {(1+\alpha)^2 \over \Lambda} \left( {1 \over \Lambda (\tilde N_+-1)} \right)^{1\over 1+\alpha}{r\over (\tilde N_+-1)(1+\alpha)+1}+o(r)$$
     From this we derive
     $$u(r) \leq 1-\left( {1\over \Lambda ( \tilde N_+-1)}\right)^{1\over 1+\alpha} r + O( r^2)$$
     and 
     $$u^{1+\alpha} \leq 1-( \alpha+1) \left( {1\over \Lambda ( \tilde N_+-1)}\right)^{1\over 1+\alpha} r + O( r^2)$$
     
     \begin{eqnarray*}
      ( |u^\prime |^\alpha u^\prime)^\prime &\geq&{ (1+\alpha)^2 \over \Lambda} r^{-1} r\left({1  \over \Lambda (N_+-1)} \right)^{1\over 1+\alpha}
     \\
     & -&{(1+\alpha)^3(\tilde N_+-1) \over \Lambda} \left({1  \over \Lambda ((\tilde N_+-1))} \right)^{1\over 1+\alpha}((\tilde N_+-1)(1+\alpha) +1)^{-1}
       +O(r)\\
       &=& 
       \left({1  \over \Lambda (N_+-1)} \right)^{1\over 1+\alpha} {(1+\alpha)^2 \over \Lambda (N_+-1)(1+\alpha)+1} +  O(r)  >0.\\
          \end{eqnarray*}
   
  \ 

Let us prove that there exists $\bar r$ so that 
$u( \bar r)=0$. If not $u>0$ and then from the equation $u^\prime$ remains $<0$. Let us define , inspired by \cite{EFQ}, \cite{D} : 
$$ y(r) = { |u^\prime |^\alpha u^\prime r^{(N-1)(1+\alpha)} \over u^{1+\alpha}}$$
 which is then always $<0$. $y$ satisfies the inequality 
 
 $$ y^\prime (r) \leq -(\alpha+1) {r^{(N-1)(1+\alpha)-\gamma} \over \Lambda} -( \alpha +1)| y|^{\alpha+2\over \alpha+1} r^{1-N}$$
  and then defining for some $r_1>0$
  $$k(r) = \int_{r_1}^r { |y|^{ \alpha+2\over \alpha+1} \over t^{N-1}} dt, $$
  
  $$ y(r) +(\alpha+1) k(r) \leq  y(r_1) -(\alpha+1) {r^{(N-1)(1+\alpha)+1-\gamma} -r_1^{(N-1)(1+\alpha)+1-\gamma} \over (N-1)(1+\alpha)+1-\gamma)  \Lambda}$$
  In particular  for some constant $c>0$  which can vary from one line to another 
  $$ y(r) \leq -c r^{(N-1)(1+\alpha)+1-\gamma} $$
   which implies that 
   $$k(r) \geq c\int_{r_1}^r { t^{(N-1)(2+\alpha) + (1-\gamma){ \alpha+2\over \alpha+1}}\over t^{N-1}} dt \geq c r^{(N-1)(1+\alpha)  +(1-\gamma){ \alpha+2\over \alpha+1}+1}. $$
   On the other hand since $k \leq -(\alpha+1)^{-1} y$ for $r$ large  one has 
   $$ k^\prime \geq  ck^{ \alpha+2 \over \alpha+1} r^{1-N}$$
    which implies 
    $${ k^\prime \over k^{ \alpha+2\over \alpha+1}} \geq c r^{1-N}$$ and after integrating between $r$ and $+\infty$, using $\lim_{r\rightarrow +\infty} k(r) = +\infty$, 
    $$ k(r) ^{1 \over \alpha+1} \leq c r^{N-2}$$
     We would then have 
     $$(N-1)(1+\alpha)  +(1-\gamma){ \alpha+2\over \alpha+1}+1\leq (N-2)(1+\alpha)$$
      which would imply 
      $\alpha+2-\gamma\leq 0$,
       a contradiction.

                                       \subsection{The stability of the principal eigenvalue and  related eigenfunctions} 
 The results of the present  section give, as a corollary,  the proof of Theorem \ref{exigamma}. 
 
  Let us start by proving the stability with respect to the $\epsilon-$regularization of the singular potential. We recall that $r_\epsilon = (r^2+\epsilon^2)^{1\over 2}$ and $\bar \lambda_\gamma^\epsilon = \bar \lambda ( F, {1\over r_\epsilon^\gamma}, B(0,1))$. 
 
  \begin{theo}\label{th3}
  One has
       $$\bar \lambda_\gamma' =\lim_{ \epsilon \rightarrow 0} \bar \lambda_\gamma^\epsilon\, .$$ 
        Furthermore,  if $\{u_\epsilon\}$ is  the sequence of the eigenfunctions associated with the eigenvalue $\bar \lambda_\gamma^\epsilon$ and satisfying  $u_\epsilon (0) = 1$, then, one can extract from $\{u_\epsilon\}$ a subsequence uniformly converging on $\overline{ B(0,1)}$ to the eigenfunction associated with $\bar \lambda_\gamma'$ which takes the value $1$ at zero.
 \end{theo}
        
           \begin{proof} 
 Let   $\{u_\epsilon\}$ be  the sequence as in the statement. Then, each $u_\epsilon$ is a ${\cal C}_{1, \alpha} $ positive radial  function  satisfying in particular
 $$
| \nabla u_\epsilon |^\alpha F(D^2u_\epsilon)+\bar \lambda_\gamma^\epsilon \frac{u_\epsilon^{1+\alpha} }{r ^\gamma}\geq 0 \quad \hbox{ in } B(0,1)\setminus \{0\}\, ,
 $$
 so that, by Theorem \ref{maxpgamma},  one has $\bar \lambda_\gamma^\epsilon\geq \bar \lambda_\gamma'$. Moreover, the sequence $\{ \bar \lambda_\gamma^\epsilon\}$ is monotone increasing with respect to $\epsilon$. Thus, we deduce
 $$
 \mu\, : = \lim_{\epsilon\to 0} \bar \lambda_\gamma^\epsilon \geq \bar \lambda_\gamma'\, .$$
On the other hand, by the monotonicity properties of radially symmetric solutions of elliptic equations, we know that $u_\epsilon'(r)\leq 0$ for $r\in [0,1]$. 
Since $${ \cal M}^+ ( D^2 u_\epsilon) \geq F( D^2 u_\epsilon)$$
we deduce that, independently of the sign of $u_\epsilon''(r)$, one has
$$ |u_\epsilon ^\prime |^\alpha  u_\epsilon^{\prime \prime} + (\tilde{N}_+-1)  |u_\epsilon ^\prime |^\alpha{u_\epsilon^\prime \over r} \geq -\frac{\bar \lambda_\gamma^\epsilon}{\lambda} u_\epsilon ^{1+\alpha}  r_\epsilon^{-\gamma} \, .$$
This implies
$$ (  |u_\epsilon ^\prime |^\alpha u_\epsilon^\prime r^{(\tilde{N}_+-1)(1+\alpha)})^\prime \geq -(1+\alpha)  \frac{\bar \lambda_\gamma^\epsilon}{\lambda} u_\epsilon^{1+\alpha}  \frac{r^{(\tilde{N}_+-1(1+\alpha) }} {r_\epsilon^\gamma}\geq - (1+\alpha)\frac{\bar \lambda_\gamma^\epsilon}{\lambda} r^{(\tilde{N}_+-1)(1+\alpha) -\gamma}\, ,$$ 
and therefore, 
by integrating, 
$$
 0\geq u_\epsilon^\prime(r)\geq - \left( \frac{\bar \lambda_\gamma^\epsilon(\alpha+1)}{\lambda ((\tilde{N}_+-1)(1+\alpha) +1-\gamma)}\right)^{1\over 1+\alpha} r^{1-\gamma\over 1+\alpha}\, .
 $$
Hence, on $\overline{B(0,1)}$, the functions $u_\epsilon$ are uniformly Lipschitz continuous if $\gamma \leq 1$, and uniformly ${2+\alpha-\gamma\over 1+\alpha}$- H\"older continuous if $\gamma>1$. In both cases, up to a subsequence, $\{u_\epsilon\}$ is uniformly converging to a continuous radial function $u\in C(\overline{B(0,1)})$ which satisfies $u(0)=1$ and 
$$ | \nabla u |^\alpha F( D^2 u) + \mu {u^{1+\alpha}  \over r^\gamma}=0.$$
Hence,  $u$ is ${\cal C}_{1, \alpha} $ and, by the standard strong maximum principle, $u$ is strictly positive in $B(0,1)$. This yields, by definition, $\mu\leq \bar \lambda_\gamma^\prime$. 
Hence, $\mu =\bar \lambda_\gamma^\prime$ and the conclusion follows from Proposition \ref{simple}.  Note that since $u^\prime<0$  for $r\neq 0$, $u$ is ${ \cal C}^2 ( B(0,1)\setminus\{0\})$. 

\end{proof}

As a consequence of the previous theorem, we finally obtain the following 

        \begin{cor}\label{delta} One has
         $$\bar\lambda_\gamma = \lim_{\delta\to 0} \bar \lambda _\gamma \left( B(0,1) \setminus \overline{B(0, \delta)}\right) =\bar \lambda_\gamma'\, .$$
          \end{cor}

      \begin{proof} 
 
 We observe that the function     $\delta \mapsto  \bar \lambda_\gamma \left( B(0,1)\setminus B(0, \delta)\right) $ is monotone increasing. Moreover, by their own definition, we have that
 $$
\bar \lambda_\gamma'\leq  \bar \lambda_\gamma \leq \bar \lambda_\gamma \left( B(0,1)\setminus B(0, \delta)\right)\, \quad \hbox{ for all } \delta\geq0\, .
 $$
On the other hand,     by Theorem \ref{th3},  for any $\eta >0$ there exists $\epsilon_0>0$ such that  
   $$ \bar \lambda_\gamma^{ \epsilon_0}  \leq \bar \lambda _\gamma' + { \eta \over 2}\, .$$
Furthermore,   by using the continuity of the principal eigenvalue with respect to the domain for equations with regular coefficients, there exists $\delta_0>0$ such that 
   $$\bar \lambda_\gamma^{\epsilon_0} ( B(0,1))\setminus B(0, \delta_0)) \leq \bar \lambda _\gamma ^{\epsilon_0}   + {\eta\over 2} \leq  \bar \lambda_\gamma' +  \eta \, .$$
Now,  since  $\epsilon \mapsto \bar  \lambda^\epsilon_\gamma \left(B(0,1))\setminus B(0, \delta_0)\right) $ decreases when $\epsilon$ decreases to zero, one gets 
  $$  \bar \lambda_\gamma \left( B(0,1)\setminus B(0, \delta_0)\right)\leq \bar \lambda_\gamma^{\epsilon_0} \left( B(0,1))\setminus B(0, \delta_0)\right) \leq \bar \lambda_\gamma' + \eta\, ,$$
   which gives the conclusion.
   
 \end{proof}

       \section{The case $\gamma = \alpha+2$}
        In this section we suppose that the operator $F$ is ${ \cal M}^+$. We will sometimes denote  $\bar \lambda_\gamma$ or $\bar \lambda_\gamma ( | \nabla \cdot |^\alpha {\cal M}^+)$ the eigenvalue. 
        Let us introduce the space of functions 
 $$\mathcal{V}= \left\{ u\in C^2([0,1])\, : u'(0)=0\, ,\ {\rm supp}(u) \hbox{ compact in } [0,1)\right\}\, ,
 $$
  endowed with the norm
 $$
 \| u\| =\left( \int_0^1 |u'|^{\alpha+2} r^{(\tilde{N}_+-1)(1+\alpha)} dr\right)^{1/\alpha+2}\, ,
 $$
and let us denote by $\mathcal{H}^1_0$  the closure of $\mathcal{V}$. 
We define    $$\bar\lambda_{var, \gamma}=\inf_{ v\in {\cal H}_o^1, \int_0^1 |v|^{\alpha+2} r^{(\tilde{N}_+-1)(1+\alpha)-\gamma} =1 } \int_0^1 |v^\prime |^{\alpha+2}  r^{(\tilde{N}_+-1)(1+\alpha)} dr$$

\begin{theo}\label{alpha+2} One has 
$$ \bar \lambda_{var, 2+\alpha}   =   \left( {(\tilde N_+-2)(\alpha+1)\over \alpha +2}\right)^{\alpha+2}:= \tau^{ \alpha+2} .  $$

\end{theo}
\begin{proof}

 We begin to prove that $\bar \lambda_{var, 2+\alpha}   \geq  \left( {(\tilde N_+-2)(\alpha+1)\over \alpha +2}\right)^{\alpha+2}$ or equivalently that 
  for any $v\in { \cal H}_o^1$
  $$\tau^{ \alpha+2}  \int _0^1{v^{2+\alpha} \over r^{2+\alpha}} r^{( \tilde N_+-1)(1+\alpha)} dr\leq \int_0^1 |v^\prime |^{2+\alpha}  r^{( \tilde N_+-1)(1+\alpha)}dr .$$
  
   For that aim we consider 
   
   $$u(r) = r^{-\tau}$$ where 
   $\tau = {(\tilde N_+-2)(\alpha+1)\over 2+\alpha}$. Then 
   $u$ satisfies for any $r>0$ in $B(0,1)$
   $$ {d\over dr} ( |u^\prime |^\alpha u^\prime r^{(\tilde N_+-1)(1+\alpha)} ) = - \tau^{ \alpha+2} u^{1+\alpha}  r^{(\tilde N_+-1)(1+\alpha)-2-\alpha}.$$
   
   Let us multiply this equation by ${|v|^{2+\alpha} \over u^{1+\alpha}}$ , for $v\in {\cal V}$. Using the definition of $u$ one observes that $\lim { |u^\prime |^\alpha u^\prime \over u^{1+\alpha}}r^{(\tilde N_+-1)(1+\alpha)} = 0$ since $\tilde N_+>2$.  \ Integrating by parts one has 
   \begin{eqnarray*}
   (\alpha+2)  \int _0^1 |u^\prime |^\alpha u^\prime |v|^\alpha v u^{-1-\alpha}v^\prime r^{(\tilde N_+-1)(1+\alpha)}  dr &-&(\alpha+1) \int _0^1 |u^\prime |^{‘\alpha+2} |v|^{\alpha +2}  u^{-2-\alpha} r^{(\tilde N_+-1)(1+\alpha)}  dr  \\
  & =& \tau^{ \alpha+2} \int_0^1 |v|^{2+\alpha}  r^{(\tilde N_+-1)(1+\alpha)-2-\alpha}.
  \end{eqnarray*}

    We denote $X = |u^\prime |^\alpha u^\prime |v|^\alpha v u^{-1-\alpha}$, $Y = v^\prime$, $ p = { \alpha+2}$ and $p^\prime = { \alpha+2 \over \alpha+1}$. Using the convexity inequality
    
     $$XY \leq {1\over p} |Y|^p+ {1\over p^\prime } |X|^{p^\prime}.$$  
     We  obtain 
     $$ \tau^{ \alpha+2}  (\alpha+2)^{-1} \int_0^1 |v|^{2+\alpha}  r^{(\tilde N_+-1)(1+\alpha)-2-\alpha}\leq {1\over \alpha+2} \int_0^1 |v^\prime |^{ \alpha+2}  r^{(\tilde N_+-1)(1+\alpha)} dr$$
     Passing to the infimum one gets 
     $ \bar \lambda_{var, 2+\alpha} \geq \tau^{\alpha+2} $. 
     
           We prove the reverse inequality.  Let for $\epsilon >0$, $\epsilon < \inf ({\tau \over 2}, {\tau \log 2 \over 4})$ 
           $ u (r) = r^{-\tau+\epsilon} ( -\log r)$. We want  to prove that  for some constant $c$
           $$ \int_0^1 |u^\prime |^{2+\alpha} r^{( \tilde N_+-1)(1+\alpha})
dr \leq \tau^{2+\alpha}   \int _0^1|u |^{2+\alpha} r^{( \tilde N_+-1)(1+\alpha)-2} (1+ c \epsilon).$$          
           We define 
           $$I= \int_0^1 (-\log r)^{2+\alpha} r^{-1+ \epsilon (2+\alpha)} dr.$$
           Note that 
            \begin{equation}\label{Itau}
             I  \geq \int _0^{1\over 2} ( \log 2)^{2+\alpha}  r^{-1+ (2+\alpha) \epsilon} dr = {C\over  \epsilon}.
             \end{equation}

             We also introduce          $$J_\epsilon  = \int_0^{e^{-1\over \tau-\epsilon}} ((\tau-\epsilon) ( -\log r)+1)^{2+\alpha} r^{-1+ \epsilon (2+\alpha)} dr, $$
            $${\rm and} \ K_\epsilon  = \int_{e^{-1\over \tau-\epsilon }}^1 ((\tau-\epsilon)  ( -\log r)+1)^{2+\alpha} r^{-1+ \epsilon (2+\alpha)} dr.$$
            
             We use the inequality for $0\leq  u\leq 1$
             $$(1+u)^{2+\alpha} \leq 1+ (2+\alpha) 2^{1+\alpha} u$$ 
               to get 
               \begin{equation}\label{Ltau}J_\epsilon  \leq \tau^{2+\alpha} I  + (2+\alpha) 2^{1+\alpha} \tau^{1+\alpha} \int_0^{e^{-1\over \tau-\epsilon}} ( -\log r)^{1+\alpha} r^{-1+ \epsilon (2+\alpha)} dr.
               \end{equation} 
               
             In the  sequel   $c(\tau, \alpha) $  denotes some positive constant depending on $\tau$ and $\alpha$ which can vary from one line to another .  Integrating by parts  the second integral in \eqref{Ltau}, one  gets 
               $$ \int_0^{e^{-1\over \tau-\epsilon }}  (-\log r)^{1+\alpha}r^{-1+2\epsilon} dr 
               \leq c(\tau, \alpha) (1+\epsilon I )\leq c( \tau, \alpha) \epsilon I $$
                by using \eqref{Itau}. From this one derives 
                $$J_\epsilon  \leq \tau^{2+\alpha}  I (1+ c( \tau, \alpha) \epsilon).$$
                
                On the other hand one has 
                
                 $$K_\epsilon  \leq  \int_{e^{-1\over \tau-\epsilon }}^1 ((\tau-\epsilon)  ( -\log r)+1)^{2+\alpha} r^{-1} dr={1\over \tau-\epsilon } \left[ \frac{((\tau-\epsilon)  ( -\log r)+1)^{3+\alpha}}{ (-(3+\alpha))}\right]_{e^{-1\over \tau-\epsilon} }^1\leq c( \tau, \alpha) \epsilon I $$
           
            We have obtained 
            $$ \int_0^1 |u^\prime |^{2+\alpha} r^{( \tilde N_+-1)(1+\alpha)}
dr  = J_\epsilon + K_\epsilon  \leq \tau^{2+\alpha} I ( 1+ c( \tau, \alpha)\epsilon) $$
            which is the desired result. 

\end{proof}
     
    In order to establish the relationship between  $\bar \lambda_{\gamma,var}$ and $\bar \lambda_\gamma$ we need to investigate on the monotonicity and convexity properties of the functions $u$ realizing the infimum in the definition of $\bar \lambda_{\gamma,var}$.

              \begin{prop}\label{propvar}
        Let $\gamma >1$. 
         Let $v_\gamma$ be in ${ \cal H}_o^1 ( B(0,1))$, $>0$  which realizes the minimum  defining $\bar \lambda_{var, \gamma}$. 
               Then 
          $v_\gamma $ is bounded, $v_\gamma^\prime <0$ and $(|v_\gamma ^\prime |^\alpha v_\gamma^\prime )^\prime >0$.
          \end{prop}
          \begin{proof}
           Since $v_\gamma $ is a minimum, it satisfies for any $u\in {\cal H}_o^1$
           \begin{equation}\label{eqmin} \int_0^1 |v_\gamma ^\prime |^\alpha v_\gamma ^\prime u^\prime r^{(\tilde{N}_+-1)(1+\alpha)} dr= \bar\lambda_{var, \gamma} \int_0^1 |v_\gamma |^{\alpha}v_\gamma  u r^{(\tilde{N}_+-1)(1+\alpha)} dr 
           \end{equation} 
         and then  in the distribution sense
           \begin{equation}\label{euler} ( |v_\gamma ^\prime |^\alpha v_\gamma ^\prime )^\prime  + (1+\alpha) {\tilde N_+-1\over r} |v_\gamma ^\prime |^\alpha v_\gamma^\prime = -\bar\lambda_{var, \gamma}v_\gamma^{1+\alpha} r^{-\gamma}.
           \end{equation} 
           In particular by the regularity properties of solutions of equations in the distribution sense, one gets $v_\gamma\in { \cal C}_{1, \alpha}$. 
                     Let us multiply the Euler equation \eqref{euler} above by $u  r ^{( \tilde N_+-1)(1+\alpha)} $ where $u\in {\cal H}_0^1$. Integrating by parts on $]\epsilon,  r[$ we obtain  
           \begin{eqnarray*}
            \int _\epsilon ^r |v_\gamma^\prime |^\alpha(s) v_\gamma^\prime (s)u^\prime(s)  s^{( \tilde N_+-1)(1+\alpha)}  ds &+& |v_\gamma^\prime |^\alpha v_\gamma^\prime (\epsilon) u( \epsilon) \epsilon ^{( \tilde N_+-1)(1+\alpha)} 
           \\
           &=& \bar\lambda_{var, \gamma} \int _\epsilon ^r  v_\gamma^{1+\alpha}(s)   u(s) s^ {( \tilde N_+-1)(1+\alpha)-\gamma} ds .
           \end{eqnarray*}
            Letting $\epsilon$ go to zero, since $u(0)$ can be chosen $\neq 0$, we obtain  from \eqref{eqmin} that 
            \begin{equation}\label{epsilon}|v_\gamma^\prime |^\alpha v^\prime_\gamma (\epsilon) \epsilon ^{( \tilde N_+-1)(1+\alpha)} \rightarrow 0.
            \end{equation}

             We now prove that $v_\gamma$ is bounded. 
            Since $v_\gamma >0$, by the equation above $|v^\prime_\gamma|^\alpha v^\prime (r) r^{(\tilde N_+-1)(1+\alpha)}\leq 0$, hence $v_\gamma$ is  decreasing $v_\gamma ( r) \leq v_\gamma (0)$.                         Thus, we have
 $v_\gamma^\prime \geq -c r^{1-\tilde N_+}$ and this implies that 
$$v_\gamma \leq c r^{2-\tilde N_+}$$
 which in turn implies 
 $$ |v_\gamma^\prime |^\alpha v_\gamma ^\prime (r) r^{ (\tilde N_+-1)(1+\alpha)} \geq -c \int v_\gamma^{1+\alpha} r^{ (\tilde N_+-1)(1+\alpha)-\gamma} \geq -c r^{2+\alpha -\gamma}$$
or equivalently  for some constant $c_1$ 
$$v_\gamma^\prime \geq -c_1 r^{{2+\alpha-\gamma\over \alpha +1} + 1-\tilde N_+}.$$
Integrating   one gets for some constant $d_1$   
 $$
 v_\gamma(r)\leq   d_1   r^{{2+\alpha-\gamma\over \alpha +1} + 2-\tilde N_+}
 $$
 Iterating the above inequalities, we obtain that for all integers $j\geq0$ such that $2-\tilde{N}_++j(2+\alpha -\gamma)<0$, there exist positive constants $c_j$ and $d_j$ satisfying
\begin{equation}\label{estj}
 v^\prime_\gamma(r)  \geq -c_j r^{1-\tilde N_++j({2+\alpha-\gamma\over 1+\alpha})}
\end{equation}
 
\begin{equation}\label{dj}
  v_\gamma(r)\leq   d_j   r^{2-\tilde{N}_++j({2+\alpha-\gamma\over 1+\alpha})}
\end{equation}
            Indeed,    by \eqref{epsilon}  inequality \eqref{estj} is true for $j=0$, and implies \eqref{dj} for $j=0$ by integrating. We now suppose  that \eqref{dj} is true for $j$ , then 
               \begin{eqnarray*}
                |v_\gamma^\prime |^\alpha v_\gamma^\prime r^{(\tilde N_+-1)(1+\alpha)} &=& -\bar\lambda_{var, \gamma} \int_0^r  v_\gamma^{1+\alpha} r^ {( \tilde N_+-1)(1+\alpha)-\gamma} \\
                & \geq& -c \int_0^r d_j^{1+\alpha}  r^{(2-\tilde N_+)(1+\alpha)+ j (2+\alpha-\gamma)}r^ {( \tilde N_+-1)(1+\alpha)-\gamma}\\
                & \geq& -c r^{(j+1)(2+\alpha-\gamma)}
                \end{eqnarray*}
                
                which implies 
                $$v_\gamma^\prime \geq -c r^{ 1-\tilde N_++ (j+1) {2+\alpha-\gamma\over 1+\alpha}} .$$
                If $1-\tilde N_++ (j+1) {2+\alpha-\gamma\over 1+\alpha}<0$, 
                 integrating this inequality we obtain \eqref{dj} for $j+1$. 
                 
                 Now, if there exists $j\in \N$ such that $2-\tilde{N}_++j{2+\alpha-\gamma\over 1+\alpha} =0$, i.e. if $\frac{(\tilde{N}_+-2)(1+\alpha)}{2+\alpha -\gamma}\in \N$, then, by integrating the estimates obtained  at the $(j-1)$-th step, we obtain
$$
v_\gamma(r)\geq -c_j \, r^{-1}\, ,\quad v_\gamma(r)\leq d_j\, (-\ln r)\, .
$$
Integrating once more, we finally deduce
$$
v'_\gamma(r)\geq -c_{j+1} (-\ln r) r^{1-\gamma}\Longrightarrow v_\gamma(r)\leq d_{j+1}=\frac{c_{j+1}}{(2-\gamma)^2}\, .
$$
On the other hand, if $\frac{(\tilde{N}_+-2)(1+\alpha)}{2+\alpha -\gamma}$ is not integer, by integrating estimates \eqref{estj} for $j=\left[ \frac{(\tilde{N}_+-2)(1+\alpha)}{2+\alpha -\gamma}\right]$, we obtain
$$
v_\gamma^\prime (r)\geq -c_{j+1} r^{1-\tilde{N}_+ +(j+1)({2+\alpha-\gamma\over 1+\alpha} )} \Longrightarrow v_\gamma(r)\leq d_{j+1}=
\frac{c_{j+1}}{2-\tilde{N}_++(j+1)({2+\alpha-\gamma\over 1+\alpha})}\, .
$$
This shows  that, in any case, $v_\gamma$ is bounded.                  We need to prove that $(|v_\gamma^\prime |^\alpha v_\gamma^{\prime})^ \prime \geq 0$. For that aim we introduce 
                 $ y =(1+\alpha) ( \tilde N_+-1)|v_\gamma^\prime |^\alpha v_\gamma^\prime + \bar \lambda_{var, \gamma}  v_\gamma ^{1+\alpha}  r^{1-\gamma}$.                  It is sufficient to prove that $y \leq 0$. One has $$y^\prime = (\tilde N_+-1) (1+\alpha)( |v_\gamma^\prime |^\alpha v_\gamma^\prime )^\prime + \bar\lambda_{var, \gamma}(1+\alpha) v_\gamma ^\alpha v_\gamma^\prime   r^{1-\gamma} 
                  + (1-\gamma) r^{-\gamma}  v_\gamma ^{1+\alpha} 
                  = (\tilde N_+-1) (1+\alpha) { -y\over r} +( \leq 0)$$
                  hence 
                   $$ (y r^{(\tilde N_+-1) (1+\alpha) })^\prime \leq 0$$
                    and then since  $ {(\tilde N_+-1) (1+\alpha) +1-\gamma}>0$,   $\lim_{r\to 0} v_\gamma(r) ^{1+\alpha} r^{1-\gamma} r^{(\tilde N_+-1) (1+\alpha) }(r)=0$ 
                                       and   since we proved above that $\lim_{r\rightarrow 0} |v_\gamma^\prime |^\alpha v_\gamma^\prime 
                    r^{ \tilde N_+-1)(1+\alpha)} = 0$,  we get $y\leq 0$, hence $(|v_\gamma^\prime |^\alpha v_\gamma^\prime )^\prime \geq 0$. 
                 \end{proof}
                   \begin{cor}\label{corvargamma}
           Let $\gamma\in ]1,2+\alpha [$. Then
           $$\bar \lambda_\gamma  (| \nabla \cdot |^\alpha \mathcal{M}^+)=  {\Lambda\over \alpha+1}  \, \bar\lambda_{\gamma, var}\, .$$
                      \end{cor}
        \begin{proof}
      By proposition \ref{propvar}, $v_\gamma$ is bounded, satisfies $\lim v_\gamma^\prime r^{ \tilde N_+-1}=0$, $v_\gamma^\prime \leq 0$ and $(|v_\gamma^\prime |^\alpha v_\gamma^{\prime})^ \prime \geq 0$, so $v_\gamma$ satisfies 
     $$ | \nabla v_\gamma |^\alpha { \cal M} ^+( D^2 v_\gamma) = -{\Lambda \bar \lambda_{var, \gamma}\over \alpha+1} {v_\gamma^{1+\alpha} \over r^\gamma}$$
        
      In particular by the definition of $\bar \lambda_\gamma$, $\Lambda ( \alpha+1)^{-1} \bar \lambda_{var, \gamma} \leq \bar \lambda_\gamma$. 
       Furthermore, analyzing the boundary condition, we get, by regularity, that  $v_\gamma(1) =0$ in the classical sense, and then, 
      by the maximum principle, since $v_\gamma = 0$ on $\partial B(0,1)$, if one had $(\alpha+1)^{-1}\Lambda \bar \lambda_{var, \gamma} <\bar \lambda_\gamma$ one would get that $v_\gamma\leq 0$, a contradiction. 
\end{proof}
      
      \begin{cor} \label{convgammavar} One has
      $$\lim_{\gamma \rightarrow 2+\alpha}  \bar \lambda_\gamma (| \nabla \cdot |^\alpha \mathcal{M}^+)=  \Lambda(\alpha+1)^{-1} \, \left( \frac{(\tilde{N}_+-2)(\alpha+1)}{\alpha+2}\right)^{2+\alpha}\, .$$
      \end{cor} 
      \begin{proof} 
           By Corollary \ref{corvargamma}, it is sufficient to prove that $ \bar \lambda_{\gamma , var}\rightarrow \bar \lambda_{2+\alpha ,var}$ as $\gamma\to 2+\alpha $. 
           
         We first observe that, by their own definition,  $ \bar \lambda_{2+\alpha , var}\leq \bar \lambda_{\gamma, var}$.  
         
On the other hand, for  any  $\epsilon>0$ there exists $v\in \mathcal{H}^1_0$, $v\neq 0$  such  that 
         $$  \int_0^1 |v^\prime |^{2+\alpha}   r^{(\tilde{N}_+-1)(1+\alpha)} dr \leq ( \bar  \lambda_{2+\alpha ,var}+ \epsilon) 
         \int_0^1 |v |^{2+\alpha}   r^{(\tilde{N}_+-1)(1+\alpha)-2-\alpha} dr\, .$$
Take $\gamma_0$ sufficiently close  to $2+\alpha $ in order that, for $2+\alpha > \gamma \geq  \gamma_0$ , 
$$ \int_0^1 |v |^{2+\alpha}   r^{(\tilde{N}_+-1)(1+\alpha)  -\gamma } dr \geq  (1-\epsilon) \int_0^1 |v |^{2+\alpha}   r^{(\tilde{N}_+-1)(1+\alpha)-2-\alpha} dr\, .$$
Thus, one has
$$\int_0^1 |v^\prime |^{2+\alpha}   r^{(\tilde{N}_+-1)(1+\alpha) } dr \leq ( \bar  \lambda_{2+\alpha ,var}+ \epsilon) (1-\epsilon)^{-1}  \int_0^1 |v |^{2+\alpha}   r^{(\tilde{N}_+-1)(1+\alpha)  -\gamma } dr dr
$$
 which yields
             $$\bar \lambda_{\gamma, var} \leq  ( \bar \lambda_{ 2+\alpha ,var}+ \epsilon) (1-\epsilon)^{-1}\, . $$
 
                \end{proof}

   \begin{cor} \label{convgammavar} One has
      $$\bar \lambda_{2+\alpha} ( | \nabla \cdot |^\alpha \mathcal{M}^+)= { \Lambda\over 1+\alpha}  \, \left( \frac{(\tilde{N}_+-2)(1+\alpha) }{2+\alpha}\right)^{2+\alpha} \, .$$
      \end{cor} 
      \begin{proof} 
           One has  for $\gamma>1$, $\gamma < 2+\alpha$ 
           $$ \bar \lambda_{2+\alpha} \leq \bar \lambda_{\gamma} =\Lambda  (\alpha+1)^{-1}\bar \lambda_{ var, \gamma} \rightarrow  (\alpha+1)^{-1}\ \Lambda\bar \lambda_{var, 2+\alpha} = \Lambda ( \alpha+1)^{-1} \tau^{\alpha+2}$$
           Furthermore taking the function 
           $ u(r) = r^{-\tau}$ one gets that $\bar \lambda_{2+\alpha} \geq   \Lambda ( \alpha+1)^{-1} \tau^{\alpha+2}$. 
           
                            \end{proof}
           \section{The case $\gamma > \alpha+2$}

           We prove Theorem \ref{gamma>}. Suppose that there exists a radial  ${ \cal C}_{1, \alpha}$ positive   supersolution, of the inequation 
           $$ | \nabla u|^\alpha  { \cal M}^-(D^2 u) + \mu u^{1+\alpha} r^{-\gamma} \leq 0$$
            with $\mu >0$ and $\gamma> \alpha+2$. 
            
            Acting as in the proof of Theorem \ref{exigamma} one easily obtains that $u^\prime$ does not change sign in a neighborhood of zero. We prove that the sign is $\leq 0$. Indeed, if we had $u^\prime \geq 0$ one would have  
             $$\Lambda (|u^\prime |^\alpha u^\prime )^\prime + \lambda {N-1)(1+\alpha)\over r} |u^\prime |^\alpha u^\prime \leq -\mu (1+\alpha) u r^{-\gamma}$$
              and then 
              $$  ( |u^\prime |^\alpha u^\prime r^{ (N_+-1)(1+\alpha)})^\prime \leq 0$$
               hence $( |u^\prime |^\alpha u^\prime r^{ (N_+-1)(1+\alpha)})$   has a limit when $r$ goes to zero, if this limit was $>0$ one would get that $u$ becomes large negative near zero. 
                Then $u^\prime \leq 0$. Coming back to the inequation, whatever is the sign of $(|u^\prime |^\alpha u^\prime )^\prime $ one has 
                
                \begin{equation}\label{ine} (|u^\prime |^\alpha u^\prime )^\prime +{ \Lambda (N-1)(1+\alpha)\over\lambda  r} |u^\prime |^\alpha u^\prime \leq -{\mu\over \Lambda}  (1+\alpha) u r^{-\gamma}
                \end{equation} 
               we then have 
              $$ ( |u^\prime |^\alpha u^\prime r^{ (N_--1)(1+\alpha)})^\prime \leq 0$$
               which implies that $ |u^\prime |^\alpha u^\prime r^{ (N_--1)(1+\alpha)}$ has a limit $\leq 0$. If it is $<0$, then there exists some constant $c>0$ so that 
               $$ u^\prime\leq -c r^{1-\tilde N_-}$$
                and integrating 
                $$ u(r) \geq c r^{2-\tilde N_-}.$$
                
                  We 
       prove that for all  integer $j\geq 0$ such  that $(\tilde N_--2)(1+\alpha)+ j( 2+\alpha-\gamma) >0$ and for $r$ sufficiently small, one has, for some $c_j>0$,
     \begin{equation}\label{eqj}u(r) \geq c_j r^{j(2+\alpha -\gamma)\over 1+\alpha}.
     .
     \end{equation}
Indeed, \eqref{eqj} holds true for $j=0$, since $u^\prime<0$ and  $u$ is positive.   Let us suppose that (\ref{eqj}) is true for 
$j$   and   that ${\tilde N_--2}+ ( j+1){2+\alpha -\gamma\over 1+\alpha} > 0$. Then, by (\ref{eqj}),   
      \begin{eqnarray*}
      -|u^\prime|^\alpha u^\prime  (r) r^{(\tilde N_--1)(1+\alpha) } & \geq& \frac{\mu}{\Lambda} c_j \int_0^{r} s^{(\tilde{N}_--1)(1+\alpha)+ j ( 2+\alpha -\gamma)-\gamma} ds \\
      &=&  \frac{\mu}{\Lambda} \frac{c_j}{(\tilde{N}_--1)(1+\alpha)+ j ( 2+\alpha -\gamma)+1-\gamma}  r^{(\tilde{N}_--1)(1+\alpha)+ j ( 2+\alpha -\gamma)+1-\gamma}\, ,
      \end{eqnarray*}
       which implies 
       \begin{equation}\label{equprime}-u^\prime  \geq c r^{j{2+\alpha-\gamma\over 1+\alpha} + {1-\gamma\over 1+\alpha}}
       \end{equation}
which, by integration, yields (\ref{eqj}) for $j+1$. 
            
 Now, let us assume  that ${ (\tilde N_--2)(1+\alpha) \over \gamma-2-\alpha} $ is not  integer. Then, using estimate \eqref{eqj} with $j=\left[ \frac{(\tilde{N}_--2)(1+\alpha)}{\gamma-2-\alpha }\right]$ jointly with \eqref{ine}, we deduce for $r_0>r>0$
 \begin{eqnarray*}
-|u^\prime (r_o)|^\alpha  u^\prime(r_0)  r_0^{{(\tilde N_--1)(1+\alpha)}}&+& |u^\prime (r) |^\alpha u^\prime (r) r^{{(\tilde N_--1)(1+\alpha) }}\\
 & \geq &\frac{ \mu\, c_j}{\Lambda\, \left((\tilde{N}_--2)(1+\alpha) + (j+1)(2+\alpha -\gamma)\right)}\left[ s^{(\tilde{N}_--2)(1+\alpha) + (j+1) (2+\alpha -\gamma)}\right]_r^{r_0}\, .
 \end{eqnarray*}
 
       Since $(\tilde{N}_--2)(1+\alpha)+ (j+1) (2+\alpha -\gamma)<0$, this yields the contradiction
$$\lim_{r\to 0}|u^\prime (r) |^\alpha u^\prime (r) r^{{(\tilde N_--1)(1+\alpha) }} =+\infty\,.$$
On the other hand, if $ \frac{(\tilde{N}_--2)(1+\alpha)}{\gamma-2-\alpha } = j+1$ is integer, then $(\tilde{N}_--2)(1+\alpha) +j(2+\alpha -\gamma)=\gamma-2-\alpha >0$, and from \eqref{eqj}  it follows that
            $$ u(r)\geq c_j r^{{\gamma+\alpha\over 1+\alpha}   -\tilde{N}_-}\, ,$$ 
hence            
            $$ u^{1+\alpha} (r) r^{( \tilde N_--1)(1+\alpha)-\gamma} \geq c_j r^{-1}\, .$$
From \eqref{ine} we then deduce, for $0<r<r_0$,
            $$-|u^\prime (r_o)|^\alpha  u^\prime(r_0)  r_0^{{(\tilde N_--1)(1+\alpha)}}+ |u^\prime (r) |^\alpha u^\prime (r) r^{{(\tilde N_--1)(1+\alpha) }} \geq \frac{\mu\, c_j}{\Lambda} \left(  \ln r_0 - \ln r\right)$$
             and we reach, also in this case, the contradiction
$$\lim_{r\to 0} u^\prime (r) r^{{\tilde N_--1}}=+\infty\,.$$

   \end{document}